\documentclass[11pt]{amsart}

\pdfoutput=1

\usepackage[text={450pt,700pt},centering]{geometry}

\usepackage{color}
\usepackage{esint,amssymb}
\usepackage{graphicx,epsfig}
\usepackage{MnSymbol}
\usepackage{mathtools}
\usepackage[colorlinks=true, pdfstartview=FitV, linkcolor=blue, citecolor=blue, urlcolor=blue,pagebackref=false]{hyperref}
\usepackage{microtype}

\usepackage{bm}
\usepackage{scalerel} 
\usepackage[font={footnotesize}]{caption}

\usepackage{mathrsfs}

\newtheorem{proposition}{Proposition}
\newtheorem{theorem}[proposition]{Theorem}
\newtheorem{lemma}[proposition]{Lemma}

\theoremstyle{remark}
\newtheorem{remark}[proposition]{Remark}

\theoremstyle{definition}

\numberwithin{equation}{section}
\numberwithin{proposition}{section}
\numberwithin{figure}{section}

\newcommand{\Z}{\mathbb{Z}}
\newcommand{\N}{\mathbb{N}}
\newcommand{\R}{\mathbb{R}}
\renewcommand{\le}{\leqslant}
\renewcommand{\ge}{\geqslant}

\newcommand{\E}{\mathbb{E}}
\renewcommand{\P}{\mathbb{P}}
\newcommand{\F}{\mathcal{F}}
\newcommand{\Zd}{\mathbb{Z}^d}
\newcommand{\Rd}{{\mathbb{R}^d}}

\newcommand{\ep}{\varepsilon}
\newcommand{\eps}{\varepsilon}

\renewcommand{\subset}{\subseteq}
\renewcommand{\fint}{\strokedint}
\newcommand{\Ll}{\left}
\newcommand{\Rr}{\right}

\newcommand{\cu}{{\scaleobj{1.2}{\square}}}

\renewcommand{\hat}{\widehat}
\newcommand{\td}{\widetilde}
\newcommand{\un}{\underline}

\newcommand{\mcl}{\mathcal}
\newcommand{\msf}{\mathsf}
\newcommand{\mbf}{\mathbf}
\newcommand{\al}{\alpha}

\newcommand{\de}{\delta}
\newcommand{\si}{\sigma}
\newcommand{\1}{\mathbf{1}}

\newcommand{\dr}{\partial}
\renewcommand{\d}{\mathrm{d}}

\renewcommand{\a}{\mathbf{a}}
\renewcommand{\b}{\mathbf{b}}
\newcommand{\ahom}{{\overbracket[1pt][-1pt]{\a}}}

\renewcommand{\O}{\mathcal{O}}

\renewcommand{\L}{\mathcal{L}}
\newcommand{\Lpot}{{\mathcal{L}^2_{\mathrm{pot}}}}
\renewcommand{\H}{\mathcal{H}}

\newcommand{\n}{\mathbf{n}}
\newcommand{\e}{\mathbf{e}}

\newcommand{\T}{{ \empty{} [ 0,1 ] ^d  }}  

\begin{document}

\title[Computing homogenized coefficients]{Computing homogenized coefficients via multiscale representation and hierarchical hybrid grids}

\begin{abstract}
We present an efficient method for the computation of homogenized coefficients of divergence-form operators with random coefficients. The approach is based on a multiscale representation of the homogenized coefficients. We then implement the method numerically using a finite-element method with hierarchical hybrid grids, which is a semi-implicit method allowing for significant gains in memory usage and execution time. Finally, we demonstrate the efficiency of our approach on two- and three-dimensional examples, for piecewise-constant coefficients with corner discontinuities. For moderate ellipticity contrast and for a precision of a few percentage points, our method allows to compute the homogenized coefficients on a laptop computer in a few seconds, in two dimensions, or in a few minutes, in three dimensions.
\end{abstract}

\author[A. Hannukainen]{A. Hannukainen}
\address[A. Hannukainen]{Department of Mathematics and Systems Analysis, Aalto University, Finland}
 \email{antti.hannukainen@aalto.fi}

\author[J.-C. Mourrat]{J.-C. Mourrat}
\address[J.-C. Mourrat]{DMA, Ecole normale sup\'erieure,
CNRS, PSL University, Paris, France}
\email{mourrat@dma.ens.fr}

\author[H. Stoppels]{H. Stoppels}
\address[H. Stoppels]{Department of Mathematics and Systems Analysis, Aalto University, Finland}
 \email{}

\date{\today}

\maketitle

%
%
%
%
%
%

\section{Introduction}

\subsection{Statement of the main results} The goal of this paper is to define, study, and implement an efficient approach to the calculation of homogenized coefficients for divergence-form operators with random coefficients. That is, we consider operators of the form $\nabla \cdot \a \nabla$, where $\a= (\a(x))_{x \in \Rd}$ is a random coefficient field on $\Rd$ taking values in the set of symmetric positive definite matrices. We assume that this random coefficient field is uniformly elliptic, $\Z^d$-stationary, and of unit range of dependence; see Subsection~\ref{ss.assump} for precise statements. 
%
%
Under these assumptions, there  exists a \emph{homogenized matrix} $\ahom$ such that the large-scale properties of the heterogeneous operator $\nabla \cdot \a(x) \nabla$ resemble those of the homogeneous operator $\nabla \cdot \ahom \nabla$. We define a multiscale method allowing to compute the homogenized matrix efficiently, and identify rigorously its rate of convergence. We then explain how to implement the algorithm in practice, using the notion of hierarchical hybrid grids, and demonstrate its performance on examples.

\smallskip 

For these numerical examples, we consider coefficient fields that are piecewise constant on a square tiling, in two dimensions, or on a cubic tiling, in three dimensions. This class of examples is particularly challenging from a computational perspective. Indeed, solutions develop singularities at the corners of the tiling which are essentially the worst possible in the class of (isotropic) coefficient fields with fixed ellipticity contrast (see Subsection~\ref{ss.rough}). Despite this, for moderate ellipticity contrast and for a precision of a few percentage points, our algorithm runs on a laptop computer and outputs a satisfactory approximation of the homogenized matrix within a few seconds in two dimensions, and within a few minutes in three dimensions. Our code is written in the Julia language and is freely available online, see the link in~\eqref{e.github}.

\smallskip

The method we explore here was introduced in \cite{efficient} in the context of discrete finite-difference equations. The main idea is to decompose the homogenized matrix into a series of terms, each of which accounting for a different length scale. The terms associated with short length scales naturally enjoy very small boundary layers and low computational effort. Those associated with larger length scales are a priori more demanding, but only appear as small correction terms in the decomposition, and can thus be computed on much smaller sample domains. Overall, this second effect more than compensates for the increase in computational effort, so that the majority of the computational time and memory is spent on the shortest length scales. A prominent feature of the method is that minimal effort is spent on the calculation of boundary layers. An additional benefit is that the method can be refined on the fly: if some calculations have already been performed and one realizes that more precision is necessary, then one does not need to throw these past calculations away and restart from scratch. 



\smallskip



We now describe this method more precisely. We fix $\xi \in \Rd$ of unit norm, and introduce the quantities that will allow us to compute $\xi \cdot \ahom \xi$. We let $v_{-1} \in H^{-1}_\mathrm{loc}(\Rd)$ be 
\begin{equation}  
\label{e.def.v-1}
v_{-1}(x) := \nabla \cdot (\a(x) \xi),
\end{equation}
and for each $k \in \N$, we define inductively $v_k \in H^1_\mathrm{loc}(\Rd)$ to be the unique stationary solution to
\begin{equation}  
\label{e.def.vk}
\Ll( 2^{-k} - \nabla \cdot \a \nabla \Rr) v_k = 2^{-k} \, v_{k-1} \qquad \text{in } \Rd.
\end{equation}
We also give ourselves a bump function $\chi \in C^\infty_c(\Rd)$ with compact support in the unit ball $B(0,1)$ and such that
\begin{equation}
\label{e.chi.int}
\int_\Rd \chi = 1.
\end{equation}
In \eqref{e.chi.int} and throughout the paper, we use the shorthand notation
\begin{equation*}  
\int_\Rd \chi = \int_\Rd \chi(x) \, \d x.
\end{equation*}
For every $r \ge 1$, we set
\begin{equation}  
\label{e.def.chir}
\chi_r(x) := r^{-d} \chi \Ll( \tfrac{x}{r} \Rr) .
\end{equation}
The following theorem is our main theoretical result. 
\begin{theorem}[Efficient approximation of $\ahom$]
For every $\ep \in (0,\frac{d-1}{2d})$, there exists a constant~$c(\ep, \|\chi\|_{H^{\Ll\lceil \frac d 2 + \frac 1 4 \Rr\rceil}(\Rd)},\Lambda,d) > 0$ such that the following holds. Let $n \in \N$, and denote
\label{t.main}
\begin{equation}
\label{e.def.rk}
r_k := 2^{n - \Ll( \frac 1 2 - \ep \Rr) k} \qquad (k \in \{0,\ldots, n\}),
\end{equation}
\begin{equation}
\label{e.def.hatsigma}
\hat \sigma_n^2 := \int_{\Rd} \Ll( -\a\xi \cdot \nabla v_0 + v_0^2 \Rr)\chi_{r_0}   + \sum_{k = 1}^n 2^k \int_{\Rd}  \Ll( v_{k-1} v_k + v_k^2 \Rr)\chi_{r_k} .
\end{equation}
For every $t \ge 0$, we have
\begin{equation}  
\label{e.main}
\P \Ll[ \Ll| \xi \cdot \ahom \xi  + \hat \sigma_n^2 - \int_\Rd (\xi \cdot \a \xi) \chi_{r_0} \Rr| \ge t  2^{-\frac{nd}{2}}  \Rr] \le 2 \exp \Ll( -c t \Rr) .
\end{equation}
\end{theorem}

Recall that we assume the law of the coefficient field $(\a(x))_{x \in \Rd}$ to be invariant under translations by vectors of $\Zd$. If we make the stronger hypothesis that the law is invariant under translations by any vector of $\Rd$, then we can replace each average against a smooth mask~$\chi_r$ in \eqref{e.def.hatsigma} by an average over the cube $(-r,r)^d$. However, under our current more restrictive assumption of invariance under translations by vectors of $\Zd$, this replacement will only work if we make sure that the side length of the box is an integer. In other words, we would need to know the identity of the underlying lattice of periods (which without loss of generality was fixed here to be~$\Z^d$) and to make sure that the domain over which we take the average contains an integer number of fundamental cells. In contrast, the formulation in Theorem~\ref{t.main} does not require that we identify the lattice of periods.

\smallskip

A result comparable to Theorem~\ref{t.main} was proved in \cite{efficient} in the context of discrete finite-difference equations. Besides the adaptation to the continuous setting, there are two main differences between the present result and the one obtained in \cite{efficient}. The first one is that the quantities on the right side of \eqref{e.def.hatsigma} are averages against a smooth mask, while only box averages could be handled in \cite{efficient}. The second and most important difference is that Theorem~\ref{t.main} gives an exponential tail estimate for the probability in \eqref{e.main}, while the result in \cite{efficient} was limited to a variance estimate. We expect the estimate \eqref{e.main} to be sharp, in the sense that we do not expect that it is possible to replace $t$ by $t^\alpha$ on the right side of \eqref{e.main} for an exponent $\alpha > 1$ that would be independent of $\ep > 0$.

\smallskip

The implementation of the method proposed in Theorem~\ref{t.main} requires the accurate calculation of $\nabla v_0$ and of $v_0, \ldots, v_n$ in $L^2$ over the progressively smaller and smaller balls $B(0,r_0), \ldots, B(0,r_n)$. As stated in~\eqref{e.def.vk}, the equation satisfied by $v_k$ is posed over the full space~$\Rd$. In practice, we can approximate these problems by selecting a sufficiently large constant~$C_{\mathrm{bl}}$ (``bl'' for ``boundary layer''), and then solving for $\td v_k \in H^1_0(B(0,r_k + C_{\mathrm{bl}} (1+n)2^\frac k 2))$ solution to
\begin{equation}  
\label{e.def.tdvk}
(2^{-k} - \nabla \cdot \a \nabla) \td v_k = 2^{-k} \, \td v_{k-1} \qquad \text{in } B(0,r_k + C_{\mathrm{bl}} (1+n)2^\frac k 2),
\end{equation}
with null Dirichlet boundary condition on $\dr B(0,r_k + C_{\mathrm{bl}} (1+n)2^\frac k 2)$, and where we have set $\td v_{-1} := v_{-1}$. The error in this approximation decays exponentially fast as we increase $C_{\mathrm{bl}}$ (this can be proved using that the Green function decays like $\exp ( -2^{-\frac k 2} |x| )$). As a rule of thumb, one should think of choosing $C_{\mathrm{bl}}$ of the order of $\sqrt{|\a|}$, where $|\a|$ is a measure of the typical size of the eigenvalues of $\a(x)$ (or, to be more specific, one can take $C_{\mathrm{bl}}$ to be of the order of $\sqrt{\Lambda}$). The additional multiplicative factor of $(1+n)$ allows for a progressive increase of the boundary layer as we increase $n$ and aim for finer and finer approximations of $\ahom$. A simple error analysis suggests that the optimal choice for the size of the boundary layer should be an affine function of $n$, and we chose it to be a multiple of $(1+n)$ for simplicity, but more refined choices can save some computations.

\smallskip

For simplicity, we implemented the version of the method described in Theorem~\ref{t.main} with $\ep = 0$. Strictly speaking, this case is not covered by Theorem~\ref{t.main}, but a modification of the arguments presented below would in this case yield \eqref{e.main} with $2^{-\frac{nd}{2}}$ replaced by $2^{-\frac{nd}{2}(1-\delta)}$, for arbitrary $\delta > 0$. (The constant $c > 0$ on the right side would then depend on $\delta$.) 

\smallskip

The main power of the method comes from the fact that it splits the problem of calculating $\xi\cdot \ahom \xi$ into multiple scales. Heuristically, the term $v_k$ (or $\td v_k$) is meant to capture information related to length scales of the order of $2^{\frac k 2}$. When $k$ is small, the elliptic problem~\eqref{e.def.tdvk} is well-conditioned and has a very small boundary layer, of essentially unit size. As $k$ is increased, the elliptic problems in \eqref{e.def.tdvk} become less well-conditioned and involve larger boundary layers. Yet, this is more than compensated by the fact that the domain of interest is rapidly shrinking. In practice, the main part of the computational effort is spent on calculating $v_0$.

\smallskip

Compared with the discrete setting of finite-difference operators investigated in~\cite{efficient}, the case of continuous differential operators considered here poses crucial new challenges from a computational perspective. With applications such as those in materials science in mind, it is natural to consider piecewise constant coefficient fields. We choose to focus more specifically on the case when the coefficient field is constant over each unit square or cube of the form $z + [0,1)^d$, where $z \in \Zd$. At least in dimension $d = 2$, this class is essentially the most difficult possible, in the sense that solutions then have the worst possible regularity properties, given the ellipticity contrast---see Subsection~\ref{ss.rough} for a more precise discussion. As a consequence of the roughness of the solutions, a ``coarsest possible'' discretization of the coefficient field into finite elements with constant coefficients would yield widely incorrect results. To wit, the algorithm as proposed here would run just fine, but it would compute the homogenized matrix associated with the particular finite-element discretization that is chosen; if the discretization is coarse, then this homogenized matrix will be \emph{far} from the homogenized matrix of the continuous operator. 

\smallskip

To remedy this problem, we thus need to rely on much finer discretizations of the coefficient field. Our method for doing so is strongly inspired by the idea of \emph{hierarchical hybrid grids} developed in \cite{ber04, ber05}. In a nutshell, the starting point is the observation that numerical schemes on fully structured grids with constant coefficients are highly efficient, both from the point of view of time and of memory usage. Unfortunately, the problem we wish to solve is not of this type, since the coefficients vary accross the domain. The idea then is to deploy a hybrid representation of the problem, using an unstructured coarse grid to represent the variations of the coefficient field on the one hand, and then proceeding to refine each coarse element in a ``fully structured'' manner. This allows for very significant gains in memory usage, which is otherwise a fundamental bottleneck, and also in execution time.

\smallskip

We did not make any effort to fine-tune the parameters of the method presented in Theorem~\ref{t.main}. We indicate here some possible directions for doing so. First, the choice to use successive powers of $2$ in \eqref{e.def.vk} can be replaced by any other real number larger than $1$, up to suitable modifications of the expression in \eqref{e.def.hatsigma}. Second, for the radii $r_k$ appearing in \eqref{e.def.hatsigma}, we simply followed the prescription of the theoretical result with $\ep = 0$, that is, $r_k = 2^{n-\frac k 2}$. A more fine-tuned method would consist in evaluating the fluctuations of integral averages on the fly and adaptively tune~$r_k$ so that the fluctuations of the average get below a certain threshold of the order of $2^{-\frac {nd}{2}}$. Finally, the requirements for accuracy are different for $v_0$, which needs to be controlled in $H^1$, and for the subsequent $v_k$'s which only need to be controlled in $L^2$. We did not try to exploit this feature either, and used approximations of the same quality for all terms.

\smallskip

Although we did not explore this possibility, we point out that the required computations can be performed in parallel in a straightforward way. For instance, instead of computing 
\begin{equation*}  
\int_{\Rd} \Ll( -\a\xi \cdot \nabla v_0 + v_0^2 \Rr) \chi_{r_0},
\end{equation*}
if one has access to $L^d$ processors, then one can compute 
\begin{equation*}  
\sum_{\ell = 1}^{L^d} \int_{\Rd} \Ll( -\a\xi \cdot \nabla v_0^{(\ell)} + (v_0^{(\ell)})^2 \Rr) \chi_{\frac{r_0}{L}},
\end{equation*}
where $(v_0^{(\ell)})_{1 \le \ell \le L^d}$ are versions of $v_0$ computed on $L^d$ independent realizations of the coefficient field. These computations can obviously be performed without any communication between processors. (If one is given a very large snapshot of a single environment, then effectively independent realizations can be obtained by considering sufficiently distant subregions of the large sample.) See also \cite{gra08} for more refined techniques allowing for the parallelization of finite-element methods with hierarchical hybrid grids.

\smallskip

For simplicity, we assume here that the coefficient field is uniformly elliptic and with a finite range of dependence. However, we expect the results presented here to hold in much greater generality. In particular, we expect that a result comparable with \eqref{e.main}, but possibly with a more slowly decaying function of $t$ on the right side, should hold whenever the local statistics of the coefficient field satisfy a central limit theorem. For more strongly correlated coefficient fields, the method is still of interest, but the choice of $r_k$ in \eqref{e.def.rk} and the term $2^{-\frac{nd}{2}}$ in \eqref{e.main} will have to be suitably modified. (This makes the development of a more adaptive algorithm particularly appealing, since such an algorithm could automatically select the optimal scalings without supervision.) Also, in view of \cite{AD,D1}, we expect that the results can be generalized to the case of perforated domains of percolation type.


\subsection{Related works}

Over the last decade, an intensive research effort has been devoted to developing theoretical quantitative results on stochastic homogenization. The multiscale representation of the homogenized coefficients forming the basis of the method is inspired by the ``renormalization'' approach to quantitative stochastic homogenization, as developed in \cite{AS,AM, AKM, AKM2}; see also \cite{review} for a gentle introduction to this line of research and \cite{AKMbook} for a monograph. A related approach based on the parabolic flow was put forward in \cite{GO5}, see also \cite[Chapter~9]{AKMbook}, and will give us the most convenient statement for us to build upon here. A different approach based on concentration inequalities was put forward in \cite{GO1,GO2,GNO,MO,GO3, GNO3}, inspired by earlier insights from statistical mechanics \cite{NS2,NS}.

\smallskip

It has been observed long ago that inappropriate boundary conditions for ``approximate cell problems'' can cause important ``resonant errors'', and initial attempts at bypassing the problem involved the notion of oversampling \cite{hw, hwc,yue-e,eh-book}. A powerful approach has been studied in \cite{blb,GO1,GO2,Gloria, GloriaH, vardecay,approx,cemracs}, based on the introduction of a small zero-order term in the equation. The method we propose here, by combining this idea with a multiscale decomposition, enables to take fuller advantage of this idea. We refer to \cite{efficient} for a detailed comparison between the single-scale and the multiscale approaches. As is shown in \cite{ADE1}, the benefits of the multiscale approach can be seen even in the setting of periodic coefficient fields, if we operate under the constraint that the lattice of periods is unknown. 

\smallskip

One alternative method for computing homogenized coefficients,  based on the idea of an ``embedded corrector problem'', is proposed in \cite{can1,can2}. Well-separated spherical inclusions are considered in the numerical examples. This allows for fairly different approaches to practical calculations than what is pursued in the present paper (and also produces solutions that are more regular than in our examples with corner discontinuities). 

\smallskip

For coefficient fields that are very similar to those we investigate numerically here, the standard representative volume method was combined with a tensor-based discretization scheme in \cite{KKO} to compute homogenized matrices, in two dimensions. 
%
%
The authors of \cite{KKO} state that their numerical approximation method displays an empirical rate of convergence in $L^2$ of $O(h^\beta)$ with $\beta \ge 3/2$, where $h$ measures the size of a discretized element. We believe that this is an artefact of pre-asymptotic effects and moderate ellipticity contrast. Indeed, for any $\alpha > 0$, solutions can develop singularities that fail to be in $H^{1+\alpha}$, provided that the ellipticity contrast is sufficiently large, and standard finite-element methods provide approximations of these singular solutions that converge in $L^2$ at a rate that is bounded below by $c\, h^{1+\alpha}$. In fact, for coefficient fields arranged in a checkerboard-type pattern in two dimensions, as considered in~\cite{KKO} and in the present paper, one can identify exactly the optimal exponent of regularity in terms of the ellipticity contrast: solutions are $H^\beta$-regular if and only if $\beta < 1+\alpha$, where~$\alpha$ is given in~\eqref{e.exponent}, as proved in \cite{pic72} and recalled in Subsection~\ref{ss.rough} below. In particular, an asymptotic rate of convergence in $L^2$ of $O(h^{3/2})$ can only be obtained for values of the ellipticity contrast~$\Lambda$ below $3 + 2 \sqrt{2}$. We also refer to the right frame of Figure~\ref{fig:ex2_mean} for an illustration of pre-asymptotic effects and slow rates of convergence, for~$\Lambda = 90$.

\smallskip

Several techniques have been explored to reduce the size of the fluctuations of estimators for the homogenized matrix. In particular, control variate techniques and the selection of special realizations of the coefficient field, called ``quasi-random structures'', have been explored, see \cite{bla16,lebris-survey} for surveys. The latter approach, inspired by \cite{wei90,zun90} and, in the context of the homogenization of elliptic operators, advocated for in \cite{quasirandom}, has recently received a spectacular theoretical foundation in~\cite{fischer}. We would find it very  interesting to investigate how these techniques can be combined with those discussed in the present paper.

\smallskip

In a different direction, several works have considered the question of designing and effectively computing certain expansions of the homogenized matrix, in situations where the random medium can be seen as a small perturbation of a reference medium. The most typical scenario is that of a homogeneous medium with a small density of inclusions \cite{maxwell,rayleigh}. We refer to \cite{kozlov, papa-survey, bm1, alb1,alb2,almog1,almog2, dl-diff, DG,almog3} for works in this area. 

\smallskip

To conclude this introduction, we mention that the homogenized matrix can also be of use as part of a modified scheme of multigrid type for computing solutions of elliptic equations with rapidly oscillating coefficients. In short, the idea is to use the homogenized operator when operating on the coarser grids \cite{efficient}.

\subsection{Organization of the paper}
In Section~\ref{s.ass}, we lay down the notation and make our standing assumptions more precise. We also clarify the meaning of being a stationary solution to \eqref{e.def.vk}, and recall the definition of the homogenized matrix. We next prove a general multiscale representation of the homogenized matrix in Section~\ref{s.multi}. By ``general'', we mean that the finite-range dependence assumption on the coefficient field is not actually used there; assuming ergodicity instead would be sufficient. This is no longer the case in Section~\ref{s.quantitative}, where we strongly leverage on the finite-range dependence assumption to obtain sharp quantitative estimates on the different terms appearing in the multiscale decomposition. This allows us to conclude the proof of Theorem~\ref{t.main}. In Section~\ref{s.hhg}, we explain how to design a finite-element multigrid algorithm using the structure of hierarchical hybrid grids. Finally, we present our numerical results in Section~\ref{s.numerics}. Our code is freely available in the GitHub repository indicated in~\eqref{e.github}.

%
%
%
%
%
%

\section{Assumptions, notation, and definition of homogenized matrix}
\label{s.ass}

\subsection{Precise statement of the assumptions}
\label{ss.assump}

We fix a constant $\Lambda \in [1,\infty)$ and an integer $d \ge 1$ throughout the paper. We denote by $\Omega$ the set of measurable mappings from $\Rd$ to the set of $d$-by-$d$ symmetric matrices which satisfy, for almost every $x \in \Rd$, 
\begin{equation}  
\label{e.unif.ell}
\forall \xi \in \Rd, \qquad \Lambda^{-1} |\xi|^2 \le \xi \cdot \a(x) \xi \le \Lambda |\xi|^2\ .
\end{equation}
For each Borel set $U \subset \Rd$, we denote by $\F_U$ the $\sigma$-algebra generated by the mappings
\begin{equation*}  
\a \mapsto \int_\Rd \phi \, \a, \qquad \phi \in C^\infty_c(U)
, 
\end{equation*}
where $C^\infty_c(U)$ denotes the set of smooth functions with compact support in $U$. We also use the shorthand $\F := \F_\Rd$. For each $y \in \Rd$, we denote by $T_y : \Omega \to \Omega$ the action of translation by $y$ on $\Omega$, which is such that, for every $x \in \Rd$,
\begin{equation*}  
T_y \a(x) = \a(x+y). 
\end{equation*}
Translations can also operate on events, that is, for every $E \in \F$ and $y \in \Rd$, we set $T_y E := \{T_y \a \ : \ \a \in E\}$.

\smallskip

We assume that we are given a probability measure $\P$ on $(\Omega,\F)$ that, in addition to \eqref{e.unif.ell}, satisfies the following properties:
\begin{itemize}  
\item stationarity with respect to $\Zd$ translations: for every $z \in \Zd$, we have
\begin{equation}  
\label{e.stationary}
\P \circ T_z = \P \; ;
\end{equation}
\item unit range of dependence: whenever two Borel sets $U,V \subset \Rd$ are at least at distance $1$ away from one another, we have that $\F_U$ and $\F_V$ are independent. 
\end{itemize}
If the latter condition was satisfied with the constant $1$ replaced by another fixed number, then we could reduce to the present setting by scaling. Similarly, if stationarity was known to hold along some lattice of $\Rd$, then we could use a change of coordinates to set it to be $\Zd$ as in our current assumption.

\subsection{General notation and function spaces}
We write $\N = \{0,1,\ldots\}$, and denote the open Euclidean ball centered at $x \in \Rd$ and of radius $r > 0$ by $B(x,r)$. We define the heat kernel at time $t > 0$ and position $x \in \Rd$ by
\begin{equation}  
\label{e.def.heatkernel}
\Phi(t,x) := (4\pi t)^{-\frac d 2} \exp \Ll( - \frac {|x|^2}{4t} \Rr) .
\end{equation}
For every Borel measurable set $U \subset \Rd$, we denote by $|U|$ the Lebesgue measure of $U$. Whenever $|U| \in (0,\infty)$, we set, for every $f \in L^1(U)$,
\begin{equation}  
\label{e.def.fint}
\fint_U f := \frac{1}{|U|} \int_U f = \frac{1}{|U|} \int_U f(x) \, \d x.
\end{equation}
For each $p \in [1,\infty)$, we define the rescaled $L^p$ norm of a measurable function $f$ as
\begin{equation*}  
\|f\|_{\un L^p(U)} := \Ll( \fint_U |f|^p \Rr) ^\frac 1 p.
\end{equation*}
For each $\ell \in \N \setminus \{0\}$, we denote by $H^\ell(U)$ the classical Sobolev space, with rescaled norm
\begin{equation}  
\label{e.def.Hell}
\|f\|_{\un H^\ell(U)} := \sum_{j = 0}^\ell |U|^{-\frac{\ell - j}{d}}\|\nabla^j f\|_{\un L^2(U)}.
\end{equation}
We denote by $H^{\ell}_0(U)$ the closure in $H^\ell(U)$ of the space $C^\infty_c(U)$ of smooth functions with compact support in $U$, and by $H^{-\ell}(U)$ the dual space to $H^\ell_0(U)$, equipped with the rescaled norm
\begin{equation}  
\label{e.def.H-ell}
\|f\|_{\un H^{-\ell}(U)} := \sup \Ll\{ \fint_U fg \ : \ g \in H^\ell_0(U) \text{ such that }\|g\|_{\un H^\ell(U)} \le 1 \Rr\}.
\end{equation}
In the expression above, we used the notation $\fint_U fg$ to denote the duality pairing between $H^{-\ell}(U)$ and $H^\ell_0(U)$ that is normalized in such a way that whenever $f$ and~$g$ are smooth, the evaluation of this duality pairing coincides with the value of the integral $\fint_U fg$.

\subsection{Notation for random variables}

In order to have concise means to express the size of random variables at our disposal, we write, for each real random variable $X$ and $s,\theta > 0$,
\begin{equation*}  
X \le \O_s(\theta) \quad \iff \quad \E \Ll[\exp \Ll( \theta^{-1} \max(X,0))^s \Rr)  \Rr] \le 2.
\end{equation*}
We also write 
\begin{equation*}  
X = \O_s(\theta) \quad \iff \quad X \le \O_s(\theta) \text{ and } - X \le \O_s(\theta).
\end{equation*}
The notation is homogeneous: we have $X \le \O_s(\theta)$ if and only if $\theta^{-1} X \le \O_s(1)$. Informally, the statement that $X \le \O_s(1)$ means that the right tail of the law of $X$ decays like $\exp(-x^s)$. The following lemma makes this precise; see \cite[Lemma~A.1]{AKMbook} for a proof.
\begin{lemma}
\label{l.bigO.vs.tail}
For every random variable $X$ and $s,\theta \in (0,\infty)$,
\begin{equation*}  
X \le \O_s(\theta) \quad \implies \quad \forall x \ge 0, \quad \P\Ll[ X \ge \theta x \Rr] \le  2 \exp \Ll( - x^s \Rr) ,
\end{equation*}
and 
\begin{equation*}  
\forall x \ge 0, \quad \P\Ll[ X \ge \theta x \Rr] \le \exp \Ll( - x^s \Rr) \quad \implies \quad X \le \O_s \Ll(2^\frac 1 s \, \theta\Rr).
\end{equation*}
\end{lemma}
The notion of $\O_s$-bounded random variables is stable under averaging, as the next lemma shows (see \cite[Lemma~A.4]{AKMbook} for a proof).
\begin{lemma}
\label{l.sum-O}
Let $s \in [1,\infty)$, $\mu$ be a measure over an arbitrary measurable space $E$, $\theta : E \to (0,\infty)$ be a measurable function and $(X(x))_{x \in E}$ be a jointly measurable family of nonnegative random variables. We have
\begin{equation*}  
\forall  x \in E, \quad X(x) \le \O_s(\theta(x)) \quad \implies \quad \int X \, d\mu \le  \O_s \Ll(\int \theta \, d \mu \Rr).
\end{equation*}
\end{lemma}
The key mechanism by which we will witness stochastic cancellations is by appealing to the following lemma.
\begin{lemma}
\label{l.barO.boxes}
For every $s \in (1,2]$, there exists a constant $C(s) < \infty$ such that the following holds. 
Let $\theta > 0$, $R \ge 1$, $\mcl Z \subset R\Z^d$, and for each $x \in \mcl Z$, let $X(x)$ be an $\mcl F(x + (-R,R)^d)$-measurable \emph{centered} random variable such that $X(x) = \O_s(\theta(x))$. We have
\begin{equation}
\label{e.barO.boxes}
\sum_{x \in \mcl Z} X(x) = \O_s\Ll(C\, \Big( \sum_{x \in \mcl Z} \theta(x)^2 \Big)^{\frac 1 2} \Rr).
\end{equation}
\end{lemma}
This lemma is a consequence of \cite[Lemmas~A.7 and A.11]{AKMbook}. Notice that when specializing Lemma~\ref{l.barO.boxes} to the case when $\theta(x) \equiv \theta$ does not depend on $x$, and denoting by $N$ the cardinality of $\mcl Z$, we can rewrite the right side of \eqref{e.barO.boxes} as $\O_s \Ll(C N^\frac 1 2 \theta\Rr)$. The term $N^\frac 12$ is consistent with the scaling of the central limit theorem.

\subsection{Definition of homogenized matrix}
We now introduce notions related to stationary random fields and solutions, and recall the definition of the homogenized coefficients in terms of correctors. 
A \emph{stationary random field} is a measurable mapping $f : \Rd \times \Omega \to \R^n$ (for some $n \in \N$) such that for every $x \in \Rd$, $z \in \Zd$ and $\a \in \Omega$,
\begin{equation*}  
f(x+z,\a) = f(x,T_z \a).
\end{equation*}
We may also simply say that the mapping $f$ is \emph{stationary}. For instance, the mapping $x \mapsto \a(x)$ itself is stationary, so the terminology is consistent with the definition in~\eqref{e.stationary}. As is standard with random objects, most of the time we do not display that a random field $f$ depends explicitly on $\a$, and simply write $f(x)$ in place of $f(x,\a)$. Extending the notation introduced in \eqref{e.def.fint},  whenever $f \in L^1_{\mathrm{loc}}(\Rd)$ is a stationary random field, we write
\begin{equation}  
\label{e.def.fintR}
\fint_\Rd f := \lim_{r \to \infty} \fint_{|x| \le r} f(x) \, \d x = \E \Ll[ \int_\T f(x) \, \d x \Rr].
\end{equation}
That this limit exists and equals the expectation on the right side follows from the ergodic theorem, see \cite{AK}.
For every $p \in [1,\infty]$, we write
\begin{equation*}  
\L^p := \Ll\{f \ : \ f \in L^p_{\mathrm{loc}}(\Rd) \text{ is a stationary random field} \Rr\},
\end{equation*}
equipped with the norm
\begin{equation*}  
\|f\|_{\L^p} := \Ll( \fint_\Rd |f|^p \Rr)^\frac 1 p.
\end{equation*}
In the case $p = \infty$, the right side is interpreted as
\begin{equation*}  
\lim_{r \to \infty} \|f\|_{L^\infty(B(0,r))},
\end{equation*}
which is also the essential supremum of the random variable $\|f\|_{L^\infty(\T)}$. We denote by $\Lpot$ the completion in $\L^2$ of the set
\begin{equation*}  
\Ll\{\nabla f \ : \ f \in C^\infty(\Rd) \text{ is a stationary random field} \Rr\}.
\end{equation*}
We also define
\begin{equation*}  
\H^1 := \Ll\{ f \in \L^2 \ : \ \nabla f \in \L^2 \Rr\},
\end{equation*}
equipped with the norm
\begin{equation*}  
\|f\|_{\H^1} := \Ll( \fint_\Rd \Ll(|f|^2 + |\nabla f|^2 \Rr)\Rr)^\frac 1 2.
\end{equation*}
Any element of $\Lpot$ can be represented as the gradient of some function $f \in H^1_{\mathrm{loc}}(\Rd)$, and such a function $f$ is defined uniquely up to a constant. However, due to the inderteminacy of this constant, the function $f$ itself may fail to be a stationary field, that is, we do not necessarily have $f \in \H^1$. We will always write elements of $\Lpot$ in the form $\nabla f$, bearing this caveat in mind. The functions $(v_k, k \in \N)$ are defined as elements of $\H^1$. The equation \eqref{e.def.vk} is interpreted, for $k = 0$, as
\begin{equation*}  
\forall w \in \H^1, \qquad \fint_\Rd \Ll(w \, v_0 + \nabla w \cdot \a \nabla v_0 \Rr) = -\fint_\Rd \nabla w \cdot \a \xi,
\end{equation*}
and, for every $k \ge 1$, as
\begin{equation}  
\label{e.weak.vk}
\forall w \in \H^1, \qquad \fint_\Rd \Ll(2^{-k} w \, v_k + \nabla w \cdot \a \nabla v_k \Rr) = \fint_\Rd 2^{-k} w \, v_{k-1}.
\end{equation}
For each $f \in \L^2$, we define the norm dual to the $\H^1$ norm by setting
\begin{equation*}  
\|f\|_{\H^{-1}} := \sup \Ll\{ \fint_\Rd f g \ : \ \|g\|_{\H^1} \le 1 \Rr\} ,
\end{equation*}
and we denote by $\H^{-1}$ the completion of $\L^2$ with respect to this norm. An example of an element of $\H^{-1}$ is $v_{-1}$, see \eqref{e.def.v-1}. By definition, for each $g \in \H^1$, the mapping 
\begin{equation}  
\label{e.mapping.to.extend}
\Ll\{
\begin{array}{rcl}  
\L^2 & \to & \R \\
f & \mapsto & \displaystyle{\fint_\Rd f g}
\end{array}
\Rr.
\end{equation}
extends to a continuous linear functional over $\H^{-1}$. Abusing notation slightly, we keep the same notation for the extension. By an integration by parts, we see that \eqref{e.weak.vk} also makes sense for $k = 0$, provided that the right side is understood in this extended sense (which is the canonical duality pairing between the spaces~$\H^{1}$ and~$\H^{-1}$).

\smallskip

Using the identity \eqref{e.def.fintR} and stationarity, one can check the following integration by parts formula: for every $f \in \H^1$ and $G \in (\H^1)^d$, we have
\begin{equation}  
\label{e.ibp}
\fint_\Rd \nabla f \cdot G = - \fint_\Rd f \, \nabla \cdot G.
\end{equation}
If we only assume $G \in (\L^2)^d$, then this formula allows to interpret $\nabla \cdot G$ as an element of $\H^{-1}$; similarly, if $f \in \L^2$, then we can interpret $\nabla f$ as an element of~$(\H^{-1})^d$.

\smallskip

The gradient of the \emph{corrector} in the direction of $\xi\in \Rd$ is the unique $\nabla \phi^{(\xi)} \in \Lpot$ that is a weak solution of
\begin{equation}  
\label{e.corrector.eq}
-\nabla \cdot \a (\xi + \nabla \phi^{(\xi)}) = 0.
\end{equation}
This equation is interpreted as
\begin{equation*}  
\text{for every } \nabla f \in \Lpot, \quad \fint_\Rd \nabla f \cdot \a(\xi + \nabla \phi^{(\xi)}) = 0.
\end{equation*}
The existence of $\nabla \phi^{(\xi)}$ can be obtained by considering, for every $\lambda \in (0,1]$, the approximation $\phi_{\lambda}^{(\xi)} \in \H^1$ which is a weak solution of the equation
\begin{equation}  
\label{e.eq.phi.lambda}
\lambda \phi_{\lambda}^{(\xi)} - \nabla \cdot \a (\xi + \nabla \phi_\lambda^{(\xi)}) = 0.
\end{equation}
It is indeed straightforward to verify that $\nabla \phi_{\lambda}^{(\xi)}$ is bounded in $\L^2$ uniformly over $\lambda \in (0,1]$, and that any weak limit must be a solution of \eqref{e.corrector.eq}. Moreover, the weak convergence of $\nabla \phi_{\lambda}^{(\xi)}$ to $\nabla \phi^{(\xi)}$ in $\L^2$ as $\lambda$ tends to $0$ can be improved to strong convergence, see e.g.\ \cite[(8.5)]{efficient}. That is, we have
\begin{equation}
\label{e.strong.L2}
\lim_{\lambda \to 0} \fint_\Rd |\nabla \phi^{(\xi)}_\lambda - \nabla \phi^{(\xi)}|^2 = 0.
\end{equation}
By definition, the homogenized matrix $\ahom$ is such that, for every $\xi \in \Rd$,
\begin{equation}  
\label{e.def.ahom}
\xi \cdot \ahom \xi = \fint_\Rd (\xi + \nabla \phi^{(\xi)}) \cdot \a(\xi + \nabla \phi^{(\xi)}).
\end{equation}
For the remainder of the paper, we will keep the unit vector $\xi \in \Rd$ fixed, and drop it from the notation: in particular, we now simply write $\phi$ in place of $\phi^{(\xi)}$.

%
%
%
%
%
%

\section{Multiscale representation}
\label{s.multi}

In this section, we give a multiscale representation of the homogenized matrix. That is, we rewrite $\ahom$ as the sum of a first term taking the form of an average of very local objects, and correction terms that involve progressively larger and larger length scales. This increase of the relevant length scale means that producing one relevant sample for the calculation of the correction term becomes progressively more difficult. Yet, the actual size of these correction terms becomes smaller and smaller, and thus fewer samples need to be averaged out in order to approximate the expected value of the quantity up to a given accuracy. Moreover, this beneficial effect more than compensates for the increase in computational effort required to obtain a single sample, and this is the main reason for the efficiency of the approach presented here.


%

\begin{proposition}[Multiscale representation of $\ahom$]
\label{p.multiscale}
Recall that we fixed $\xi \in \Rd$ of unit length, and that $v_{-1}, v_0, v_1,\ldots$ are defined in \eqref{e.def.v-1}-\eqref{e.def.vk}. For each $n \in \N$, the limit
\begin{equation}  
\label{e.def.Dn}
D_n := \lim_{\lambda \to 0} \fint_\Rd v_n \Ll( \lambda - \nabla \cdot \a \nabla \Rr)^{-1} v_n  
\end{equation}
exists and is finite. Moreover,
\begin{equation}  
\label{e.multiscale}
\xi\cdot \ahom \xi = \fint_\Rd \xi \cdot \a \xi -  \sum_{k = 0}^{n} 2^k \fint_\Rd \Ll( v_{k-1} v_k + v_k^2 \Rr) -  D_n.
\end{equation}
\end{proposition}
\begin{remark}  
In the summand indexed by $k = 0$ on the right side of \eqref{e.multiscale}, we have
\begin{equation*}  
\fint_\Rd v_{-1} v_0 = -\fint_\Rd   \a\xi \cdot \nabla v_0,
\end{equation*}
and the left side of the identity above is interpreted as the duality pairing between $\H^{-1}$ and $\H^1$, as explained below \eqref{e.mapping.to.extend}. All the other terms on the right side of \eqref{e.multiscale} involve functions in $\L^2$ and are thus understood as in \eqref{e.def.fintR}.
\end{remark}
\begin{remark}  
One can show in great generality (using only the ergodicity of the coefficient field instead of the short-range dependence assumption) that 
\begin{equation*}  
\lim_{n \to \infty} D_n = 0.
\end{equation*}
We do not provide the argument for this fact here. The interested reader can reconstruct it from the quantitative analysis of this term provided in the next section; see also \cite[Theorem~5.1]{efficient}. 
\end{remark}
\begin{proof}[Proof of Proposition~\ref{p.multiscale}]
By \eqref{e.corrector.eq}, we have
\begin{equation}  
\label{e.test.phi}
\fint_\Rd \nabla \phi \cdot \a (\xi+\nabla\phi) = 0.
\end{equation}
Using also \eqref{e.def.ahom}, we deduce that
\begin{equation}  
\label{e.split.terms}
\xi \cdot \ahom \xi = \fint_\Rd (\xi+\nabla \phi) \cdot \a (\xi+\nabla \phi) = \fint_{\Rd} \xi\cdot \a \xi - \fint_\Rd \nabla \phi \cdot \a \nabla \phi.
\end{equation}
For each $\lambda > 0$, we define the resolvent operator
\begin{equation*}  
R_\lambda : \Ll\{
\begin{array}{rcl}  
\H^{-1} & \to & \H^1 \\
f & \mapsto & (\lambda - \nabla \cdot \a \nabla)^{-1} f.
\end{array}
\Rr.
\end{equation*}
The function $u = R_\lambda f$ is interpreted as the unique element of $\H^1$ such that, for every $v \in \H^1$,
\begin{equation*}  
\fint_{\Rd} \Ll( \lambda u v + \nabla u \cdot \a \nabla v \Rr) = \fint_\Rd fv,
\end{equation*}
the right side of this identity being understood as explained below \eqref{e.mapping.to.extend}. 
For every $\lambda, \mu > 0$, we have the resolvent formula
\begin{equation}  
\label{e.resolvent.formula}
R_\lambda = R_\mu + (\mu - \lambda) R_\lambda R_\mu.
\end{equation}
In particular, we have $R_\lambda R_\mu = R_\mu R_\lambda$. Moreover, the operator $R_\lambda$ is self-adjoint, in the sense that for every $f,g \in \H^{-1}$,
\begin{equation}  
\label{e.selfadjoint}
\fint_\Rd f R_\lambda g = \fint_\Rd g R_\lambda f.
\end{equation}
By \eqref{e.test.phi} and \eqref{e.strong.L2}, we have
\begin{equation*}  
\fint_\Rd \nabla \phi \cdot \a \nabla \phi 
= - \fint_\Rd \a \xi \cdot \nabla \phi 
 = -  \lim_{\lambda \to 0} \fint_\Rd \a \xi \cdot \nabla \phi_\lambda 
=  \lim_{\lambda \to 0} \fint_\Rd v_{-1} R_\lambda v_{-1} .
\end{equation*}
The completion of the proof will follow from this identity and a repeated application of the resolvent formula. To start with, given any family of numbers $\lambda, \mu_0,\ldots, \mu_n \in (0,\infty)$, an inductive argument based on the identity \eqref{e.resolvent.formula} yields that
\begin{equation}  
\label{e.recursive.resolv}
R_\lambda =  \Ll(\sum_{k = 0}^n (\mu_0 - \lambda) \, \cdots \, (\mu_{k-1} - \lambda) R_{\mu_0} \, \cdots \, R_{\mu_k}\Rr)
+ (\mu_0 - \lambda) \, \cdots \, (\mu_{n} - \lambda) R_{\mu_0} \, \cdots \, R_{\mu_n} R_\lambda.
\end{equation}
For the summand with $k = 0$, the product $(\mu_0 - \lambda) \, \cdots \, (\mu_{k-1} - \lambda)$ appearing above is interpreted as being $1$. Note that, for every $k \in \N$, we have 
\begin{equation}
\label{e.vk.resolvent}
v_k = 2^{-k} R_{2^{-k}} v_{k-1}.
\end{equation}
We define recursively
\begin{equation*}  
v_{-1,\lambda} := v_{-1}, \qquad \forall k \in \{0,\ldots,n\}, \ v_{k,\lambda} := (2^{-k}-\lambda) R_{2^{-k}} v_{k,\lambda}.
\end{equation*}
For each $\lambda \in (0,2^{-n})$, we now apply the formula \eqref{e.recursive.resolv} with the choice
\begin{equation*}  
(\mu_0,\ldots,\mu_{2n}) = (1,1,2^{-1}, 2^{-1}, \ldots, 2^{-n}, 2^{-n}),
\end{equation*}
and also use the commutation between resolvents and the symmetry \eqref{e.selfadjoint} to obtain that
\begin{equation*}  
\fint_\Rd v_{-1} R_\lambda v_{-1} = \sum_{k = 0}^n (2^{-k}-\lambda)^{-1} \fint_\Rd \Ll( v_{k-1,\lambda} v_{k,\lambda} + v_{k,\lambda}^2 \Rr) + \fint_\Rd v_{n,\lambda} \, R_{\lambda} v_{n,\lambda}.
\end{equation*}
Note that $v_{k,\lambda}$ is a scalar multiple of $v_k$, and that this scalar tends to $1$ as $\lambda$ tends to $0$. Hence, the left side and each summand in the sum indexed by $k$ on the right side of the identity above converges as $\lambda$ tends to $0$. It follows that the limit
\begin{equation*}  
\lim_{\lambda \to 0} \fint_\Rd v_{n,\lambda} \, R_{\lambda} v_{n,\lambda} = \lim_{\lambda \to 0} \fint_\Rd v_{n,} \, R_{\lambda} v_{n} =: D_n
\end{equation*}
is well-defined and finite, and using also \eqref{e.split.terms}, that the formula \eqref{e.multiscale} holds.
\end{proof}

%
%
%
%
%
%

\section{Quantitative estimates}
\label{s.quantitative}

Recall that we have fixed a vector $\xi \in \Rd$ of unit norm throughout the paper. The proof of Theorem~\ref{t.main} relies on estimates on the solution $u$ of the initial-value problem
\begin{equation}
\label{e.pde.u}
\Ll\{
\begin{array}{ll}
\dr_t u - \nabla \cdot \a \nabla u = 0 & \quad \text{in } (0,\infty) \times \Rd, \\
u(0,\cdot) = \nabla \cdot \a \xi & \quad \text{on } \Rd.
\end{array}
\Rr.
\end{equation}
The study of this problem was initiated in \cite{vardecay} where suboptimal estimates were derived. The sharp exponent of decay in time was obtained in \cite{GNO}, with polynomial moments controlled. With a very different proof, the stochastic integrability of this result was improved to almost Gaussian tails in \cite{GO5}. A variation of this argument is exposed in \cite[Chapter~9]{AKMbook}. 

\begin{theorem}[\cite{GO5}]
\label{t.estim.u}
(1) For every $\sigma \in (0,2)$, there exists $C(\sigma,\Lambda,d) < \infty$ such that for every $t \ge 1$ and $x \in \Rd$,
\begin{equation*}
|u(t,x)| \le \O_\sigma \Ll( C t^{-\frac 1 2 - \frac d 4} \Rr) .
\end{equation*}

(2) For every $\de > 0$, there exist $\si(\delta,d) > 2$ and $C(\de,\Lambda,d) < \infty$ such that for every $t \ge 1$ and $x \in \Rd$,
\begin{equation}  
\label{e.estim.u}
\Ll|u(t,x)\Rr| \le \O_\si \Ll( C t^{-\frac 1 2 - \frac d 4 + \de} \Rr) .
\end{equation}
\end{theorem}
We will only use the first part of Theorem~\ref{t.estim.u} once, in the course of the proof of Proposition~\ref{p.remainder}, in the form of the $L^2$ estimate
\begin{equation}  
\label{e.L2.estimate}
\fint_\Rd u^2(t,\cdot) \le C t^{-1 - \frac d 2}.
\end{equation}
In order to conclude for exponentially decaying tails as in the statement of Theorem~\ref{t.main}, it is crucial to be able to choose an exponent $\sigma \ge 2$ in the estimate on the size of $u$ (and for convenience, we will in fact choose $\sigma > 2$); this is the main motivation for stating the second part of Theorem~\ref{t.estim.u}. The first part of Theorem~\ref{t.estim.u} matches the results found in \cite{GO5} and \cite[Theorem~9.1]{AKMbook}. In order to obtain the second part of the statement as a consequence, we can use the following basic deterministic estimate, a proof of which can be found in \cite[Lemma~9.2]{AKMbook}.
\begin{lemma}[Deterministic bounds on $u$]
\label{l.deterministic.u}
There exists $C(\Lambda,d) < \infty$ such that for every $t > 0$,
\begin{equation*}  
\|u(t,\cdot)\|_{L^\infty(\Rd)} + t^\frac 1 2 \|\nabla u(t,\cdot)\|_{L^\infty(\Rd)} \le C t^{-\frac 1 2}.
\end{equation*}
\end{lemma}
\begin{proof}[Proof of part (2) of Theorem~\ref{t.estim.u}]
Let $\sigma > 2$ and $\tau \in (0,2)$ be exponents that will be fixed in terms of $\de > 0$ and the dimension~$d$ in the course of the proof. By the first part of the theorem, there exists a constant $C(\tau,\Lambda,d) < \infty$ such that for every $t \ge 1$ and $x \in \Rd$, 
\begin{equation*}
|u(t,x)| \le \O_\tau \Ll( C t^{-\frac 1 2 - \frac d 4} \Rr) .
\end{equation*}
Explicitly, this means that
\begin{equation*}  
\E \Ll[ \exp \Ll( \Ll(C^{-1} t^{\frac 1 2 + \frac d 4}|u(t,x)|\Rr)^\tau \Rr)  \Rr] \le 2.
\end{equation*}
It follows from Lemma~\ref{l.deterministic.u} that the random variable $|u(t,x)|$ is bounded, uniformly over $t \ge 1$ and $x \in \Rd$. Hence, for a constant $C(\tau,\sigma,\Lambda,d) < \infty$,
\begin{equation*}  
\E \Ll[ \exp \Ll( C^{-\sigma} t^{\tau\Ll(\frac 1 2 + \frac d 4\Rr)}|u(t,x)|^\sigma \Rr)  \Rr] \le 2,
\end{equation*}
that is,
\begin{equation*}  
\Ll|u(t,x)\Rr| \le \O_\si \Ll(C t^{-\frac \tau \si \Ll( \frac 1 2 + \frac d 4 \Rr)}  \Rr) .
\end{equation*}
This is \eqref{e.estim.u} with $\delta$ given by
\begin{equation*}  
\delta = \Ll( 1 - \frac \tau \sigma \Rr) \Ll( \frac 1 2 + \frac d 4 \Rr) .
\end{equation*}
Since $\sigma > 2$ and $\tau < 2$ can be chosen arbitrarily close to $2$, any exponent $\delta > 0$ can be represented in this way.
\end{proof}
We also record the following useful lemma allowing to localize the dependency of $u$ on the coefficient field; see \cite[Lemma~9.4]{AKMbook} for a proof.
\begin{lemma}[Localization of~$u$]
\label{l.SG.loc.u}
There exist a constant $C(\Lambda,d) < \infty$ and, for each $r \in [2,\infty)$, $t \in (0,r^2]$ and $x \in \Rd$, an $\mcl F(B(x,r))$-measurable random variable $u'(r,t,x)$ such that
\begin{equation}  
\label{e.loc.u}
\Ll|u(t,x) - u'(r,t,x) \Rr| \le C t^{-\frac 1 2} \exp \Ll( - \frac{r^2}{Ct}\Rr),
\end{equation}
and 
\begin{equation}  
\label{e.loc.nabla.u}
\Ll|\nabla u(t,x) - \nabla u'(r,t,x) \Rr| \le C t^{-1}\exp \Ll( - \frac{r^2}{Ct}\Rr).
\end{equation}
\end{lemma}
The function $v_k$ can be represented as a time integral of the function $u(t,\cdot)$, and the main contribution to this integral is for $t \simeq 2^k$. The previous results concerning the function $u$ can thus be translated into information on $v_k$.
\begin{proposition}[quantitative bounds on $v_k$]
\label{p.quantitative.vk}
(1) There exists $C(\Lambda,d) < \infty$ such that for every $k \in \N$,
\begin{equation}  
\label{e.deterministic.v}
\|v_k\|_{L^\infty(\Rd)} 
\le C 2^{-\frac k 2}.
\end{equation}

(2) There exist $C(\Lambda,d) < \infty$ and, for each $r \in [2,\infty)$, $k \in \N$  and $x \in \Rd$, an $\F(B(x,r))$-measurable random variable $v_k(r,x)$ such that 
\begin{equation}
\label{e.loc.v}
\Ll| v_k(x) - v_k(r,x)  \Rr| \le C 2^{-\frac k 2} \exp \Ll( -C^{-1} 2^{-k} r^2 \Rr) ,
\end{equation}
and
\begin{equation}
\label{e.loc.nabla.v}
\Ll|\nabla v_k(x) - \nabla v_k'(r,x) \Rr| \le C 2^{-k}  \exp \Ll( - C^{-1} 2^{-k} r^2\Rr).
\end{equation}

(3) For every $\de > 0$, there exist $\si(\delta,d) > 2$ and $C(\de,\Lambda,d) < \infty$ such that for every $k \in \N$ and $x \in \Rd$,
\begin{equation}
\label{e.estim.v}
|v_k(x)| \le \O_\si \Ll( C 2^{-k \Ll( \frac 1 2 + \frac d 4 - \delta \Rr) } \Rr) .
\end{equation}
\end{proposition}
\begin{proof}
We decompose the proof into three steps. 

\smallskip

\emph{Step 1.}
Transfering information on $u(t,\cdot)$ onto information on the $v_k$'s relies on the observation that, for every $\lambda > 0$ and $f \in \L^2$,
\begin{equation}  
\label{e.resolvent.semigroup}
R_\lambda f = (\lambda - \nabla\cdot \a \nabla)^{-1} f = \int_0^{+\infty} e^{-\lambda t} P(t) f \, dt,
\end{equation}
where $P(t) = \exp \Ll( t \, \nabla \cdot \a \nabla \Rr)$ is the semigroup associated with the evolution operator $\partial_t - \nabla \cdot \a \nabla$. This identity can be extended to the case $f = \nabla \cdot \a \xi$, and we thus have in particular that
\begin{equation}  
\label{e.formula.v0}
v_0 = \int_0^{\infty} e^{-t} \, u(t,\cdot) \, \d t.
\end{equation}
Since integrals of the form of \eqref{e.resolvent.semigroup} or \eqref{e.formula.v0} will be iterated multiple times, it is convenient to rewrite them using probabilistic notation. That is, denoting by $T^{(\lambda)}$ an exponential random variable of parameter $\lambda$ which is independent of any other quantity in the problem, and by $E$ the expectation over this random variable only, we can rewrite \eqref{e.resolvent.semigroup} in the form
\begin{equation}  
\label{e.resolvent.Tlambda}
\lambda R_\lambda f = E \Ll[P(T^{(\lambda)}) f\Rr].
\end{equation}
Denote by $(T_k)_{k \in \N}$ a family of independent exponential random variables, of respective parameters $(2^{-k})_{k \in \N}$, and for every $k \in \N$, set
\begin{equation}  
\label{e.def.Sk}
S_k := \sum_{j = 0}^k T_j.
\end{equation}
We keep denoting by $E$ the expectation over these random variables. Recalling \eqref{e.vk.resolvent} and using the semigroup property of $(P(t))_{t \ge 0}$, we deduce that for every $k \in \N$, 
\begin{equation}  
\label{e.formula.vk}
v_k = E \Ll[ P(S_k) v_{-1} \Rr] = E \Ll[u(S_k,\cdot) \Rr].
\end{equation}
In view of this relation, we can now transfer information about $u$ onto $v_k$ provided that we have some information on the typical behavior of $S_k$. As a useful guide for the intuition, we remark that
\begin{equation*}  
E[S_k] = \sum_{j = 0}^k 2^{j} = 2^{k+1} - 1,
\end{equation*}
so heuristically, we hope that any bound on $u(t,\cdot)$ transfers into a bound on~$v_k$ after the substitution of $t$ by $2^k$.

\smallskip

\emph{Step 2.} We prove \eqref{e.estim.v}. In view of Theorem~\ref{t.estim.u} and \eqref{e.formula.vk}, we need to know that~$S_k$ is rarely much smaller than $2^{k}$. The following result is shown in \cite[(5.12)]{efficient}: for every $\beta > 0$, there exists a constant $C(\beta) < \infty$ such that for every $k \in \N$,
\begin{equation}
\label{e.left.tail.Sk}
E \Ll[ (1+S_k)^{-\beta} \Rr] \le C 2^{-k\beta}.
\end{equation}
By Theorem~\ref{t.estim.u} and Lemma~\ref{l.sum-O}, we have, for every $x \in \Rd$,
\begin{equation*}  
E \Ll[\Ll|u(S_k,x)\Rr| \, \1_{\{S_k \ge 1\}} \Rr] \le \O_s \Ll( C E \Ll[(1+S_k)^{-\frac 1 2 - \frac d 4 + \de}\Rr]  \Rr) \le \O_s \Ll( C 2^{-k \Ll( \frac 1 2 + \frac d 4 - \de \Rr) } \Rr) .
\end{equation*}
To control the behavior of this term on the event $S_k \in [0,1]$, we use that the $T_j$'s are nonnegative and independent to write, for every $s \le 1$,
\begin{align} 
\label{e.estim.Skle}
P[S_k \le s] 
& \le \prod_{j = 0}^k P[T_j \le s]
 \le \prod_{j = 0}^k \Ll( 1-\exp \Ll( -2^{-j }s \Rr)  \Rr) 
 \le \prod_{j = 0}^k 2^{-j}s = 2^{-\frac{k(k+1)}{2}} s^{k+1}.
\end{align}
This and Lemma~\ref{l.deterministic.u} imply that
\begin{multline}  
\label{e.estim.shorttime}
E \Ll[ \Ll|u(S_k,x)\Rr| \, \1_{\{S_k \le 1\}} \Rr] 
 \le \sum_{j = 0}^{\infty} E \Ll[ \Ll|u(S_k,x)\Rr| \, \1_{\{2^{-j-1} < S_k \le 2^{-j}\}} \Rr] 
\\ \le C  2^{-\frac{k(k+1)}{2}} \sum_{j = 0}^\infty 2^{\frac{j}{2}} 2^{-j(k+1)}
 \le C 2^{-\frac{k(k+1)}{2}}.
\end{multline}
This is largely sufficient to complete the proof of the estimate in \eqref{e.estim.v}. The proof of \eqref{e.deterministic.v} is similar, only simpler, using Lemma~\ref{l.deterministic.u} in place of Theorem~\ref{t.estim.u}.

\smallskip

\emph{Step 3.} We prove \eqref{e.loc.v}. Recalling the notation $u'(r,t,x)$ introduced in Lemma~\ref{l.SG.loc.u}, we define, for each $r \in [2,\infty)$ and $x \in \Rd$, the $\mcl F(B(x,r))$-measurable random variable
\begin{equation*}  
v_k(r,x) := E \Ll[ u'(r,S_k,x) \1_{\{S_k \le r^2\}}\Rr] .
\end{equation*}
We need an upper bound for the probability of the event that $S_k$ is large. We obtain this by writing, for every $s \ge 0$, 
\begin{multline}  
\label{e.exp.decay}
P \Ll[ S_k \ge s \Rr] \le \E \Ll[ \exp \Ll( 2^{-(k+1)} \Ll( S_k - s \Rr)  \Rr)  \Rr] 
 =\prod_{j = 0}^{k} \frac{2^{-j}}{2^{-j} - 2^{-(k+1)}} \, \exp \Ll( -2^{-(k+1)} s \Rr) 
\\
 =\prod_{j = 1}^{k+1} \frac{1}{1 - 2^{-j}} \, \exp \Ll( -2^{-(k+1)} s \Rr) 
 \le C \exp \Ll( -2^{-(k+1)} s \Rr) .
\end{multline}
We now decompose the error into
\begin{align*}  
 \Ll|v_k(x) - v_k(r,x) \Rr| 
&  \le \E \Ll[ |u'(r,S_k,x) - u(S_k,x)| \, \1_{\{S_k \le r^2\}} \Rr] + \E \Ll[ |u(S_k,x)| \, \1_{\{S_k > r^2\}} \Rr] 
\\
&  \le  \sum_{j = 0}^{\lfloor 2^{-k} r^2 \rfloor } \E \Ll[ |u'(r,S_k,x) - u(S_k,x)| \, \1_{\{j2^{k} \le S_k \le (j+1)2^k\}} \Rr] + \E \Ll[ |u(S_k,x)| \, \1_{\{S_k > r^2\}} \Rr],
\end{align*}
and analyze each of these terms in turn. 
By Lemma~\ref{l.SG.loc.u} and \eqref{e.exp.decay}, we have
\begin{multline*}  
 \sum_{j = 0}^{\lfloor 2^{-k} r^2 \rfloor } \E \Ll[ |u'(r,S_k,x) - u(S_k,x)| \, \1_{\{j2^{k} \le S_k \le (j+1)2^k\}} \Rr] 
\le C 2^{-\frac k 2} \sum_{j = 0}^{\lfloor 2^{-k} r^2 \rfloor }   \exp \Ll( -\frac{r^2}{C j 2^{k}} -\frac{j}{2}\Rr)
\\
 \le C 2^{-\frac k 2}\exp \Ll( -\frac{r^2}{C 2^k} \Rr) .
\end{multline*}
Moreover, by Lemma~\ref{l.deterministic.u} and \eqref{e.exp.decay}, we have
\begin{equation*}  
\E \Ll[ |u(S_k,x)| \, \1_{\{S_k > r^2\}} \Rr] \le C r^{-1} \exp \Ll( -2^{-(k+1)} r^2 \Rr) .
\end{equation*}
Combining these estimates yields \eqref{e.loc.v}. The proof of \eqref{e.loc.nabla.v} is similar, except that we appeal to \eqref{e.loc.nabla.u} instead of \eqref{e.loc.u}.
\end{proof}
We denote
\begin{equation*}  
w_0 := -\a\xi \cdot \nabla v_0 + v_0^2, \qquad \forall k \in \N \setminus \{0\}, \ w_k := v_{k-1} v_k + v_k^2.
\end{equation*}
For every $k \in \N$, we have that $w_k \in \L^1$, and by Proposition~\ref{p.multiscale}, for $D_n$ as in~\eqref{e.def.Dn}, 
\begin{equation}  
\label{e.multiscale.rewrite}
\xi \cdot \ahom \xi = \fint_\Rd \xi \cdot \a \xi - \sum_{k = 0}^n 2^k \fint_\Rd w_k - D_n.
\end{equation}
We next estimate the size of the remainder term $D_n$.

\begin{proposition}[Remainder estimate]
\label{p.remainder}
There exists $C(\Lambda,d) < \infty$ such that for every $n \in \N$,
\begin{equation*}  
0 \le D_n \le C 2^{-\frac{nd}{2}}.
\end{equation*}
\end{proposition}
\begin{proof}
We keep denoting by $S_k$ the random variable defined in \eqref{e.def.Sk}, and we let~$S_k'$ be an independent copy of $S_k$. We also let $T^{(\lambda)}$ be an independent exponential random variable of parameter $\lambda$, and denote the expectation with respect to these random variables by $E$. We recall that the introduction of these random variables allows us for convenient representations such as those in \eqref{e.resolvent.Tlambda} and \eqref{e.formula.vk}. Combining these representations with the definition of $D_k$ in \eqref{e.def.Dn}, we obtain that
\begin{equation*}  
D_k = \lim_{\lambda \to 0} \lambda^{-1} \fint_\Rd E \Ll[ P(S_k) v_{-1} P(T^{(\lambda)}+S_k')v_{-1}\Rr] .
\end{equation*}
Since $P(t)$ is self-adjoint in $\L^2$, we have
\begin{align*}  
\fint_\Rd E \Ll[ P(S_k) v_{-1} P(T^{(\lambda)}+S_k')v_{-1}\Rr] 
& = \fint_\Rd E \Ll[ \Ll(P \Ll( \frac{S_k + S_k' + T^{(\lambda)}}{2} \Rr) v_{-1}\Rr)^2 \Rr] 
\\
& = \fint_\Rd E \Ll[ u^2\Ll( \frac{S_k + S_k' + T^{(\lambda)}}{2} ,\cdot\Rr) \Rr] ,
\end{align*}
and thus
\begin{equation}  
\label{e.positivity.Dk}
D_k = \int_0^{+\infty} \fint_\Rd E \Ll[ u^2\Ll( \frac{S_k + S_k' + t}{2} ,\cdot\Rr) \Rr]  \, \d t \ge 0. 
\end{equation}
In order to estimate the integral over $t \in [0,1]$, we use independence to get that
\begin{equation*}  
P \Ll[ S_k + S_k' \le s \Rr] \le \Ll(P \Ll[ S_k \le s \Rr] \Rr)^2,
\end{equation*}
and then proceed as in \eqref{e.estim.Skle} and \eqref{e.estim.shorttime}. For the remaining part, we use \eqref{e.L2.estimate} and the triangle inequality to deduce that 
\begin{equation*}  
D_k  \le C \int_1^\infty E \Ll[ (S_k + S_k' + t)^{-1 -\frac d 2} \Rr] \, \d t 
  \le 
C \int_1^\infty E \Ll[ \Ll(S_k + \frac t 2\Rr)^{-1 -\frac d 2} \Rr] \, \d t .
\end{equation*}
Integrating in $t$ and then appealing to \eqref{e.left.tail.Sk}, we obtain
\begin{equation*}  
D_k \le C E \Ll[ (S_k + 1)^{-\frac d 2} \Rr]  \le C 2^{-\frac{kd}{2}},
\end{equation*}
as announced.
\end{proof}

In the next proposition, we replace the global averages $\fint_\Rd$ appearing in \eqref{e.multiscale.rewrite} by averages against a heat kernel mask. Recall the definition of $\Phi$ in \eqref{e.def.heatkernel}.
\begin{proposition}[CLT cancellations]
\label{p.spatav}
For every $\de > 0$, there exist $\si(\delta,d) > 1$ and $C(\delta,\Lambda,d) < \infty$ such that for every $k \in \N\setminus \{0\}$, $x \in \R$ and $s > 0$,
\begin{equation}  
\label{e.spatav}
\Ll|\int_\Rd w_k(y) \Phi(s,x-y) \, \d y - \fint_{\Rd} w_k  \Rr| 
\le  \O_\si \Ll( C 2^{-k \Ll( 1 + \frac d 2 - \de \Rr) } \Ll(  \frac{s}{2^k \log^2(2+k)} + 1 \Rr) ^{-\frac d 4} \Rr) .
\end{equation}
For $k =0$, the same estimate holds with the additional restriction $s \ge 1$.
\end{proposition}
\begin{proof}
We first record the following elementary observation: for every $\si,\theta > 0$ and random variables $X$ and $Y$, we have 
\begin{equation}
\label{e.mult.Os}
\mbox{$|X| \le \O_{2\si} (\theta)$ and $|Y| \le \O_{2\si}(\theta)$} \quad \implies \quad |XY| \le \O_\si (\theta^2).
\end{equation}
Indeed, this follows from
\begin{align*}  
\E \Ll[ \exp \Ll( \Ll(\theta^{-2} |XY| \Rr)^{\si}\Rr)  \Rr] 
& 
\le \E \Ll[ \exp \Ll( \frac 1 2 \Ll(\theta^{-1} |X|\Rr)^{2\si} + \frac 1 2 \Ll(\theta^{-1} |Y|\Rr)^{2\si} \Rr)  \Rr] 
\\
& \le \E \Ll[ \exp \Ll( \Ll(\theta^{-1} |X|\Rr)^{2\si}\Rr)\Rr]^\frac 1 2 \,  \E \Ll[ \exp \Ll( \Ll(\theta^{-1} |Y|\Rr)^{2\si} \Rr)  \Rr]^\frac 1 2 \le 2.
\end{align*}
We decompose the rest of the proof into three steps.

\smallskip



\emph{Step 1.} We set 
\begin{equation}  
\label{e.def.rk'}
r_k' := 2^\frac k 2 \log(2+k).
\end{equation}
In this step, we observe that the statement \eqref{e.spatav} is valid when $\sqrt s \le r_k'$. Indeed, for $k \ge 1$, the statement \eqref{e.spatav} with $\sqrt s \le r_k'$ follows from \eqref{e.estim.v}, \eqref{e.mult.Os} and Lemma~\ref{l.sum-O}. For $k = 0$, we also need a bound on $\nabla v_0$, which is provided by the following deterministic estimate: 
there exists a constant $C(\Lambda,d) < \infty$ such that for every $x \in \Rd$,
\begin{equation}  
\label{e.deterministic.nablav0}
\|\nabla v_0\|_{L^2(x+\cu_0)} \le C.
\end{equation}
This estimate is a consequence of the Caccioppoli inequality and \eqref{e.deterministic.v}.

\smallskip

\emph{Step 2.}
We reformulate \eqref{e.spatav} into an equivalent form that will be more convenient for the analysis. 
For every $k \in \N$, we denote 
\begin{equation*}  
\td w_k := w_k - \E[ w_k].
\end{equation*}
We show that it suffices to prove Proposition~\ref{p.spatav} for $\sqrt s \ge r_k'$ and with \eqref{e.spatav} replaced by
\begin{equation}
\label{e.spatav.td}
\int_\Rd \td w_k(y) \Phi(s,x-y) \, \d y  =  \O_\si \Ll( C 2^{-k \Ll( 1 + \frac d 2 - \de \Rr) } \Ll( \frac{\sqrt s}{r_k'} +1 \Rr) ^{-\frac d 2} \Rr) .
\end{equation}
For every $k \ge 1$, the mapping $x \mapsto \E[w_k(x)]$ is $\Z^d$-periodic, and by~\eqref{e.estim.v}, it is uniformly bounded by $C2^{-k \Ll( 1 + \frac d 2 - \delta \Rr)}$. In the case $k = 0$, the mapping $x \mapsto \E[w_0(x)]$ is $\Z^d$-periodic and in $L^2([0,1]^d)$, as follows from~\eqref{e.deterministic.v} and the fact that $v_0 \in \mcl H^1$. Hence, for every $k \in \N$, there exists a constant $C(k,\Lambda,d) < \infty$ such that for every $x \in \Rd$ and $s \ge 1$, 
\begin{equation*}  
\Ll|\int_\Rd \E[w_k(y)] \Phi(s,x-y) \, \d y - \fint_\Rd w_k \Rr| \le C 2^{-k \Ll( 1 + \frac d 2 - \delta \Rr)}\exp \Ll( -C^{-1} s \Rr) .
\end{equation*}
See for instance \cite[Exercise~3.7]{AKMbook} for a proof. As a consequence, the statements~\eqref{e.spatav} and \eqref{e.spatav.td} are equivalent, up to an adjustment of the constant~$C$.

\smallskip

\emph{Step 3.}
Without loss of generality, it suffices to prove \eqref{e.spatav.td} for $x = 0$. 
For each $r \ge r_k'$, we define
\begin{equation*}  
\td w_k'(r,x) := w_k'(r,x) - \E \Ll[ w_k'(r,x) \Rr] .
\end{equation*}
By \eqref{e.loc.v}-\eqref{e.loc.nabla.v}, there exists $C(\Lambda,d) < \infty$ such that for every $k \in \N$ and $r \ge r_k'$,
\begin{equation}  
\label{e.whatmoredoyouwant}
\Ll|\td w_k(x) - \td w_k'(r,x) \Rr| \le C 2^{-100dk} \exp \Ll( -C^{-1} 2^{-k} r^2 \Rr) .
\end{equation}
In this step, we leave aside the case $k = 0$ and show that there exists $C(\de,\Lambda,d) < \infty$ such that for every $k \ge 1$, $r \ge r_k'$ and $s > 0$,
\begin{equation}  
\label{e.spatav.tdr}
\int_\Rd \td w_k'(r,x) \Phi(s,x) \, \d x  = \O_\si \Ll( C 2^{-k \Ll( 1 + \frac d 2 - \de \Rr) } \Ll( \frac {s} {r^2} + 1  \Rr) ^{-\frac d 4}  \Rr) .
\end{equation}
We denote the cube of side length~$r$ centered at the origin by
\begin{equation*}  
\cu_r := \Ll( -\frac r 2, \frac r 2 \Rr)^d.
\end{equation*}
By \eqref{e.whatmoredoyouwant} and the same argument as in Step 1 of this proof, we may assume that $s \ge r^2$. 
We decompose the left side of \eqref{e.spatav.tdr} into
\begin{equation*}  
\int_\Rd \td w_k'(r,x) \Phi(s,x) \, \d x = \sum_{z \in r \Zd} \int_{z + \cu_r} \td w_k'(r,x) \Phi(s,x) \, \d x .
\end{equation*}
By \eqref{e.estim.v}, \eqref{e.mult.Os} and Lemma~\ref{l.sum-O}, there exists $\sigma > 1$ such that, for each $z \in \Zd$,
\begin{equation*}  
\int_{z + \cu_r} \td w_k'(r,x) \Phi(s,x) \, \d x = \O_\sigma \Ll( C 2^{-k \Ll( 1 + \frac d 2 - \delta \Rr)} \|\Phi(s,\cdot)\|_{L^1(z + \cu_r)}  \Rr) .
\end{equation*}
By Lemma~\ref{l.barO.boxes}, we obtain that
\begin{align*}  
\sum_{z \in r \Zd}\int_{z + \cu_r} \td w_k'(r,x) \Phi(s,x) \, \d x = \O_\sigma \Ll( C 2^{-k \Ll( 1 + \frac d 2 - \de \Rr)} \Ll(\sum_{z \in r\Zd} \|\Phi(s,\cdot)\|_{L^1(z + \cu_r)}^2\Rr)^\frac 1 2 \Rr) .
\end{align*}
We conclude that \eqref{e.spatav.tdr} holds by observing that, since $s \ge r^2$,
\begin{equation*}  
\sum_{z \in r\Zd} \|\Phi(s,\cdot)\|_{L^1(z + \cu_r)}^2  \le C r^d s^{-\frac d 2}. 
\end{equation*}

\smallskip

\emph{Step 4.} In this step, we show that there exists $C(\de,\Lambda,d) < \infty$ such that for every $k \ge 1$,  $r \ge r_k'$ and $s > 0$,
\begin{equation}  
\label{e.spatav.tdr.2r}
\int_\Rd \Ll(\td w_k'(2r,x) - \td w_k'(r,x)\Rr) \Phi(s,x) \, \d x  
= \O_\si \Ll( C 2^{-k \Ll( 1 + \frac d 2 - \de \Rr) } \Ll( \frac {s} {r^2} + 1  \Rr) ^{-\frac d 4}  \exp \Ll( -\frac{r^2}{C 2^k} \Rr) \Rr) .
\end{equation}
The argument is similar to that of the previous step, only simpler, using only \eqref{e.whatmoredoyouwant} and not requiring any appeal to \eqref{e.estim.v}.

\smallskip

\emph{Step 5.} We complete the proof. It is clear from \eqref{e.whatmoredoyouwant} that 
\begin{equation*}  
\td w_k(x) = \lim_{r \to \infty} \td w'_k(r,x).
\end{equation*}
We thus decompose $\td w_k(x)$ into
\begin{equation*}  
\td w_k(x) = \td w_k'(r_k',x) + \sum_{j = 0}^{+\infty} \Ll( \td w_k'(2^{j+1}r_k',x) - \td w_k'(2^jr_k',x) \Rr) .
\end{equation*}
Applying \eqref{e.spatav.tdr} to the first term, \eqref{e.spatav.tdr.2r} to each of the summands, and summing the result, we obtain \eqref{e.spatav.td} for $k \ge 1$. In the case $k = 0$, the same reasoning applies, using also the deterministic estimate on $\nabla v_0$ provided in \eqref{e.deterministic.nablav0}, and the estimate \eqref{e.loc.nabla.v} to localize the dependency of this term on the coefficient field.
\end{proof}

We have now obtained almost optimal information on the behavior of $w_k$ when tested against the heat kernel. Since we want to understand the behavior of this field against an arbitrary mask, we now upgrade this information into an $H^{-\ell}$ estimate using the following lemma, which is a rescaled version of~\cite[Remark~D.6]{AKMbook}. In order to keep the presentation of the argument as simple as possible, we only state this lemma for $L^2$-based Sobolev spaces with integer-valued regularity exponents. Recall the definitions of the rescaled $H^\ell$ and $H^{-\ell}$ norms in \eqref{e.def.Hell} and \eqref{e.def.H-ell}.
\begin{lemma}[Sobolev norm from heat-kernel convolutions]
\label{l.sobolev.app}
For every $\ell \in \N$, there exists a constant $C(\ell,d) < \infty$ such that for every $f \in H_\mathrm{loc}^{-\ell}(\Rd)$ and $r > 0$, we have
\begin{equation*}  
\|f\|_{\un H^{-\ell}(B(0,r))}^2 \le C r^{-d} \int_{\Rd} \exp \Ll( -\frac {|x|}{r} \Rr) \int_0^{r^2} s^{\ell-1} |f \ast \Phi(s,\cdot)|^2 (x) \, \d s \, \d x.
\end{equation*}
\end{lemma}

Combining this lemma with Proposition~\ref{p.spatav} yields the following estimate.
\begin{lemma}[Sobolev norm for $w_k$]
\label{l.sobolev}
For every $\de > 0$, there exist $\si(\de,d)  > 1$ and, for every $\ell \in \N$ satisfying $\ell > \frac d 2$, a constant $C(\de,\ell,\Lambda,d) < \infty$ such that for every $k \in \N$ and $r \ge 1$, we have
\begin{equation*}  
r^{-\ell} \Ll\|w_k - \fint_\Rd w_k\Rr\|_{\un H^{-\ell}(B(0,r))} \le \O_\si \Ll( C 2^{-k \Ll( 1 + \frac d 2 - \de \Rr) } \, \Ll(\frac{r}{2^\frac k 2 \log(2+k)}+1\Rr)^{-\frac d 2}\Rr) .
\end{equation*}
\end{lemma}
\begin{proof}
For convenience, we set $f := w_k - \fint_\Rd w_k$, and use the notation $r_k'$ introduced in \eqref{e.def.rk'}. We first consider the case $k \in \N \setminus \{0\}$. For $r \le r_k'$, by Proposition~\ref{p.spatav}, we have  for every $x \in \Rd$ that
\begin{align*}  
\int_0^{r^2} s^{\ell-1} |f \ast \Phi(s,\cdot)|^2 (x) \, \d s & = \O_{\si/2} \Ll( C 2^{-2k \Ll( 1 + \frac d 2 - \de \Rr) }\int_0^{r^2} s^{\ell-1}   \, ds\Rr) 
 = \O_{\si/2} \Ll(  C 2^{-2k \Ll( 1 + \frac d 2 - \de \Rr) } r^{2\ell} \Rr).
\end{align*}
For $r \ge r_k'$, we have instead that
\begin{align*}  
\int_0^{r^2} s^{\ell-1} |f \ast \Phi(s,\cdot)|^2 (x) \, \d s & = \O_{\si/2} \Ll( C 2^{-2k \Ll( 1 + \frac d 2 - \de \Rr) }\int_0^{r^2} s^{\ell-1}   \Ll( \frac{\sqrt{s}}{r_k'} +1 \Rr) ^{-d}  \, ds\Rr) \\
& = \O_{\si/2} \Ll(  C 2^{-2k \Ll( 1 + \frac d 2 - \de \Rr) } r^{2\ell} \Ll( \frac{r}{r_k'} \Rr) ^{-d} \Rr),
\end{align*}
where we used that $\ell > \frac d 2$ and $r \ge r_k'$ for the second equality. We then obtain the result by an application of Lemma~\ref{l.sobolev.app}. The case $k = 0$ is obtained in the same way, except that we treat the integral over $s \in [0,1]$ separately using the gradient estimate in \eqref{e.deterministic.nablav0}.
\end{proof}

We are now ready to complete the proof of Theorem~\ref{t.main}.
\begin{proof}[Proof of Theorem~\ref{t.main}]
Recall that we denote by $\chi \in C^\infty_c(\Rd)$ a smooth bump function of unit mass with compact support in the unit ball $B(0,1)$, and that we write $\chi_r := r^{-d} \chi(r^{-1} \, \cdot)$. By the definition of the rescaled $H^\ell$ norm in \eqref{e.def.Hell}, for every $\ell \in \N$ and $r \ge 1$, we have
\begin{equation*}  
\|\chi_r\|_{\un H^\ell(B(0,r))} = r^{-\ell} \|\chi\|_{\un H^\ell(B(0,1))}.
\end{equation*}
We select $\ell$ to be the smallest integer such that $\ell > \frac d 2$ (that is, $\ell = \Ll\lceil \frac d 2 + \frac 1 4 \Rr\rceil$), and write
\begin{equation*}  
\Ll|\int_{\Rd} w_k \chi_r - \fint_\Rd w_k \Rr| \le r^{-\ell}  \|\chi\|_{\un H^\ell(B(0,1))} \Ll\| w_k - \fint_\Rd w_k \Rr\|_{\un H^{-\ell}(B(0,r))} .
\end{equation*}
An application of Lemma~\ref{l.sobolev} thus yields, for each $\delta >0$, that there exists $\sigma(\de,d) > 1$ and a constant $C(\de,\|\chi\|_{\un H^\ell(\Rd)},\Lambda,d) < \infty$ such that for every $k \in \N$ and $r\ge 1$, 
\begin{equation*}  
\Ll|\int_{\Rd} w_k \chi_r - \fint_\Rd w_k \Rr| \le \O_\si \Ll( C 2^{-k \Ll( 1 + \frac d 2 - \de \Rr) } \, \Ll(\frac{r}{2^\frac k 2 \log(2+k)}+1\Rr)^{-\frac d 2}\Rr) .
\end{equation*}
We fix $\ep \in \Ll( 0, \frac{d-1}{2d} \Rr)$, $n \in \N$, and recall from \eqref{e.def.rk} the notation
\begin{equation*}  
r_k := 2^{n- \Ll( \frac 1 2 - \eps \Rr)k }.
\end{equation*}
Substituting $r$ with $r_k$ in the previous display, we obtain that
\begin{align*}  
\Ll|\int_{\Rd} w_k \chi_{r_k} - \fint_\Rd w_k \Rr| & 
\le \O_\si \Ll( C 2^{-k \Ll( 1 + \frac d 2 - \de \Rr) } \, \Ll(\frac{2^{n-k+\ep k}}{\log(2+k)}+1\Rr)^{-\frac d 2} \Rr) \\
& \le \O_\si \Ll( C 2^{-\frac{nd}{2}-k \Ll(1-\de + \frac{\ep d}{2}\Rr)} \, \log^\frac d 2(2+k)\Rr) .
\end{align*}
We select $\de := \frac{\ep d}{4} > 0$, 
so that
\begin{equation*}  
\sum_{k = 0}^n 2^k \Ll|\int_{\Rd} w_k \chi_{r_k} - \fint_\Rd w_k \Rr| \le \O_\si \Ll( C 2^{-\frac {nd}{2}} \Rr) .
\end{equation*}
In view of Lemma~\ref{l.bigO.vs.tail} and of the fact that $\si > 1$, this implies the existence of a constant $c(\ep,\|\chi\|_{\un H^\ell(\Rd)},\Lambda,d) < \infty$ such that for every $n \in \N$ and $t \ge 0$,
\begin{equation*}  
\P \Ll[ \sum_{k = 0}^n 2^k \Ll|\int_{\Rd} w_k \chi_{r_k} - \fint_\Rd w_k \Rr| \ge t 2^{-\frac{nd}{2}}\Rr] \le 2 \exp \Ll( -ct \Rr)  .
\end{equation*}
A similar but simpler argument shows that
\begin{equation*}  
\P \Ll[ \Ll|\int_{\Rd} (\xi\cdot \a\xi) \chi_{r_0} - \fint_\Rd \xi \cdot \a \xi \Rr| \ge t 2^{-\frac{nd}{2}}\Rr] \le 2 \exp \Ll( -ct^2 \Rr).
\end{equation*}
Combining these estimates with \eqref{e.multiscale.rewrite} and Proposition~\ref{p.remainder} completes the proof of Theorem~\ref{t.main}. 
\end{proof}
%
%
%
%
%
%

\section{Hierarchical hybrid grids}
\label{s.hhg}

In this section, we explain our strategy for the numerical approximation of solutions of elliptic equations. For definiteness, given a coefficient field $\a(x)$, and a domain $U \subset \Rd$, we consider the problem of computing an approximation of the solution $u \in H^1_0(U)$ of the equation
\begin{equation}  
\label{e.model.equation}
\Ll\{
\begin{aligned}
& -\nabla \cdot \a(\xi+\nabla u) = 0 & \text{ in } U,
\\
& u = 0 & \text{ on } \dr U.
\end{aligned}
\Rr.
\end{equation}
Since generalizations such as the addition of lower-order terms or non-zero boundary conditions pose no particular additional difficulty, we will not discuss these further.
 In the first subsection, we observe the necessity to opt for highly refined discretized approximations to the continuous equation. We then explain efficient ways to compute these highly refined approximations. Our approach is in line with the earlier work \cite{ber04,ber05} and based on hierarchical hybrid grids. That is, we start from an unstructured coarse mesh and refine it in a self-similar way a number of times; we then exploit this piecewise-structured hierarchical construction extensively at every step of the algorithm (assembly of the finite-element matrix, matrix-vector products, restriction and interpolation operators).

\subsection{Roughness of solutions}
\label{ss.rough}

In many practical instances, the heterogeneity of the coefficient field is due to the fact that the material of interest is a mixture of several different types of substances: see for instance the library of images at \cite{image-lib}. In view of this, we focus on the case of piecewise-constant coefficient fields.


\begin{figure}
\begin{center}
\includegraphics[scale=0.33]{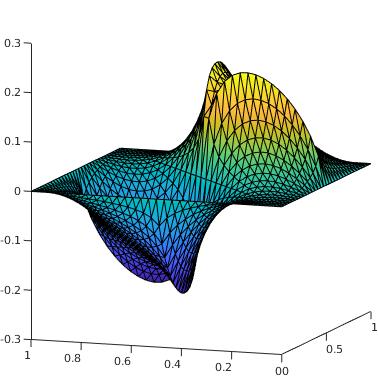}  \hspace{.1cm}
\includegraphics[scale=0.33]{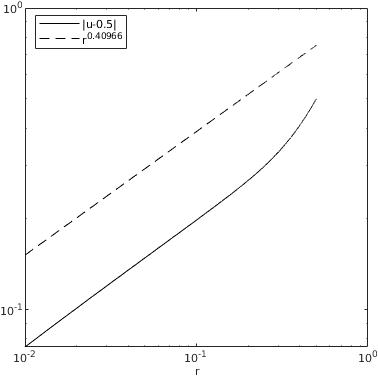} \hspace{.1cm}
\includegraphics[scale=0.33]{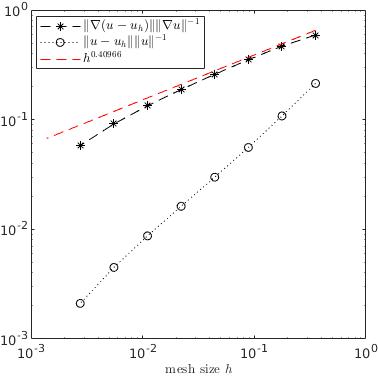}
\end{center}
\caption{
Left: the solution $u$ of \eqref{e.eq.4squares}-\eqref{e.coefs.4squares} minus the affine function $x \mapsto x_1 + x_2$, with initial coarse mesh refined five times. Middle: the solution along the line $x_1 = x_2$ for $x_1 > 0$, compared with the function $r^{0.40966}$, on a logarithmic scale. Right: the (approximate) relative error, in $H^1$ and in $L^2$ respectively, for the problem in \eqref{e.eq.4squares}-\eqref{e.coefs.4squares}. Successive dots on a given line correpond to successive refinements of the triangular mesh, starting from a coarse mesh of $8$ triangles. }
\label{f.2x2error_singular}
\end{figure}

\smallskip

In this case, the discontinuities of the coefficient field compound the difficulties inherent to solving equations with rapidly oscillating coefficients. In order to measure the extent of these difficulties, consider the problem of approximating the solution $u \in H^1([-1,1]^2)$ to 
\begin{equation}  
\label{e.eq.4squares}
\Ll\{
\begin{aligned}
& -\nabla \cdot \a(x) \nabla u = 0  & \text{in } \ [-1,1]^2,
\\
& u(x) = x_1 + x_2 & \quad \text{on } \dr \Ll( [-1,1]^2 \Rr) ,
\end{aligned}
\Rr.
\end{equation}
where the coefficient field $\a(x)$ is given by
\begin{equation}  
\label{e.coefs.4squares}
\a(x) = 
\Ll|
\begin{array}{ll}  
9 \, \mathrm{Id} & \quad \text{if } x \in [-1,0]^2 \cup [0,1]^2, \\
\mathrm{Id} & \quad \text{otherwise}. \\
\end{array}
\Rr.
\end{equation}
We start from a coarse mesh made of 8 triangles of equal sizes (two triangles in each of the translates of $[0,1]^2$). We then refine a given mesh by subdividing each triangle into 4 smaller triangles, adding a new vertex at the midpoint of each edge. We consider multiple iterations of this refinement procedure, and for each level of refinement, we compute the associated finite-element solution, using piecewise affine elements. 
The approximation of the solution minus the affine function $x = (x_1,x_2) \mapsto x_1 + x_2$ after five levels of refinement is represented on the left frame of Figure~\ref{f.2x2error_singular}. The rough behavior of the solution near the origin is clearly visible. We also display the value of the solution along the line $x_1 = x_2$ on the middle frame of Figure~\ref{f.2x2error_singular}---see below around \eqref{e.exponent} for the prediction of the exponent $0.4\ldots$ appearing there. The right frame of Figure~\ref{f.2x2error_singular} displays the relative error, measured in $H^1$ and in $L^2$ respectively, compared with the true solution. In order for the relative error to be below $10\%$ in the~$H^1$ norm, it is necessary to use at least six levels of refinement. 
At six levels of refinement, the linear system that needs to be solved already involves $2^{15}$ unknowns.

\smallskip

We can understand the roughness of the solution theoretically in a precise way. We consider more generally the situation when the coefficient field is given by
\begin{equation}  
\label{e.coefs.4squares.Lambda}
\a(x) = 
\Ll|
\begin{array}{ll}  
\Lambda \, \mathrm{Id} & \quad \text{if } x \in [-1,0]^2 \cup [0,1]^2, \\
\mathrm{Id} & \quad \text{otherwise}, \\
\end{array}
\Rr.
\end{equation}
for some $\Lambda \in [1,\infty)$. 
A blow-up analysis near the origin suggests to look for solutions in the unit ball $B(0,1)$ of the form $r^\alpha f(\theta)$, where $r \ge 0$ and $\theta \in [0,2\pi)$ are the standard polar coordinates: $x_1 = r \cos \theta$ and $x_2 = r \sin \theta$. Denoting
\begin{equation*}  
a(\theta) := 
\Ll|
\begin{array}{ll}  
\Lambda & \quad \text{if } \theta \in \Ll[ 0,\tfrac \pi 2 \Rr] \cup \Ll[ \pi, \tfrac {3\pi}{2} \Rr] , \\
1 & \quad \text{otherwise}, \\
\end{array}
\Rr.
\end{equation*}
we find that the smallest exponent $\alpha > 0$ such that $r^\alpha f(\theta)$ is a solution in $B(0,1)$ for some function $f$ is given by
\begin{equation*}  
\alpha^2 = \inf \Ll\{ \frac{\int_0^{2\pi} (f')^2 a}{\int_0^{2\pi} f^2 a} \ : \ f \in H^1_{\mathrm{per}}(\Ll[0,2\pi\Rr]) \ \text{ s.t. } \int_0^{2\pi}  f a = 0\Rr\}.
\end{equation*}
Moreover, $r^\alpha f(\theta)$ is indeed a solution of the equation in $B(0,1)$ when $f$ is the unique minimizer of the variational problem above. The value of this exponent was computed in \cite{pic72}: it is
\begin{equation}  
\label{e.exponent}
{\alpha} = \frac{4}{\pi} \arctan \Ll( \frac{1}{\sqrt{\Lambda}} \Rr) .
\end{equation}
Notice that the function $r^\alpha f(\theta)$ belongs to $H^{1+\alpha-\ep}(B(0,1))$ for every $\ep > 0$, but does not belong to $H^{1+\alpha}(B(0,1))$. We therefore expect a finite-element scheme with elements of size $h$ to provide an approximation in $H^1$ at a precision of the order of~$h^\alpha$ (and at precision of the order of $h^{1+\alpha}$ in $L^2$). The particular case we investigated numerically corresponds to $\Lambda = 9$, which gives
\begin{equation}  
\label{e.exponent.1.9}
\alpha = \frac{4}{\pi} \arctan \Ll( \frac{1}{3} \Rr) \simeq 0.4096655294\ldots
\end{equation}
In fact, it was shown in \cite{pic72} that the exponent in \eqref{e.exponent} is the smallest possible exponent for H\"older regularity one can get if one allows for abritrary coefficient fields which are everywhere a multiple of the identity and satisfy the ellipticity condition
\begin{equation}  
\label{e.ellipt}
\mathrm{Id} \le \a(x) \le \Lambda \, \mathrm{Id}.  
\end{equation}
In this sense, coefficient fields that are piecewise constant on a checkerboard structure are \emph{worst possible} from the point of view of regularity (and therefore of difficulty of numerical approximation).	
For general coefficient fields satisfying \eqref{e.ellipt} but not necessarily being a multiple of the identity matrix at each point, it was shown in~\cite{pic72} that the smallest possible exponent for H\"older regularity is $\alpha = \Lambda^{-\frac 1 2}$. An explicit coefficient field satisfying \eqref{e.ellipt} and admitting a solution of the form $r^{\Lambda^{-\frac 1 2}} f(\theta)$ was first given in \cite{meyers}. This exponent governs the rate of convergence of the finite-element approximation as the mesh is successively refined: for instance, for an ellipticity contrast of $\Lambda = 100$, we cannot hope for an asymptotic convergence rate better than $h^{0.1}$ in general, and no better than $h^{0.127\ldots}$ in the case of the coefficient field in \eqref{e.coefs.4squares.Lambda}. 
The situation is even worse in dimension $d = 3$, at least from a theoretical point of view. Indeed, to the best of our knowledge, it is an open question to show that when $d = 3$ (or for any $d \ge 3$), the regularity exponent can be bounded from below by a negative power of the ellipticity contrast. 

\subsection{Number of unknowns}

We are ultimately interested in solving elliptic equations with \emph{random} coefficients. In order to calculate the homogenized matrix, we will need to average over large domains, so as to tame the fluctuations of the coefficient field. As a toy example, consider the problem of calculating the standard average of the coefficient field, denoted by $\fint_\Rd \a$ above, see \eqref{e.def.fintR}. By the scaling of the central limit theorem, in order to measure this quantity within a precision $\delta > 0$, we need to average over at least $C \de^{-2}$ unit cells. Similarly, as was shown in \cite[Proposition~1.1]{efficient}, it is impossible to compute an approximation of the homogenized matrix at precision $\delta$ if one observes only $o(\de^{-2})$ unit cells (the statement in \cite{efficient} is written for finite-difference equations, but the proof applies essentially verbatim to the continuous setting).

\smallskip

Roughly speaking, if we want to compute $\ahom$ within a precision of, say, $10\%$, we are bound to have to examine at least of the order of $10^2$ unit cells. In two dimensions, if the mesh we use is refined six times as described in the previous subsection, this means that we must be facing problems involving of the order of $2^{15} \cdot 10^2 \simeq 3 \cdot 10^6$ unknowns. Notice that each further refinement of the mesh multiplies this number by $4$, and that reducing the size of the fluctuations by a factor of $2$ also multiplies this number by $4$. Finally, this rough estimation hides multipicative constants that may be large. (On the other hand, the random coefficient fields we investigate numerically in the next section are not made of a systematic periodic repetition of the worst-case coefficient field in~\eqref{e.coefs.4squares}, and this will mitigate the difficulty somewhat.)

\subsection{Motivations for hybrid methods}

The upshot of the previous subsections is that we ought to be able to solve for elliptic problems with many degrees of freedom. As is well-known, the numerical approximation of elliptic equations in domains with simple geometry and with constant coefficients can be performed very efficiently using a variety of techniques, including the geometric multigrid method (see \cite{FFT} for several benchmarks). Indeed, for equations with constant coefficients, stencil-based operations can replace the need to assemble and store the finite-element matrix. Moreover, the data can be organized locally in agreement with the underlying geometry and accessed in a consistent way, resulting in few integer operations and highly optimized usage of the processor cache memory. 

\smallskip

For more complex geometries or varying coefficients, completely unstructured approaches can be used instead. In this case, the problem of storing the finite-element matrix in memory becomes a major limitation. Moreover, data access becomes highly unpredictable and requires more integer operations, two factors that cause a dramatic drop in performance \cite{ber04,ber05}. 

\smallskip

Following \cite{ber04,ber05}, we seek to remedy this problem by using a hybrid approach. The idea, called \emph{Hierarchical hybrid grids} in \cite{ber04,ber05}, is to proceed as in the completely unstructured case on the coarse mesh, but then rely on structured techniques within each constant-coefficient patch. This approach has multiple advantages. Firstly, we only need to assemble and store the finite-element matrix associated with the coarsest mesh. Similarly, we do not need to store the full computational grid in memory. This results in large gains in memory usage, which is otherwise the main limiting factor on the computing architectures we use.  Moreover, we store a vector of the finite-element space in a bi-dimensional array indexed by the identity of the coarse element and then the position within it. This allows to obtain efficiency gains similar to those observed in the completely structured case, in particular regarding fast matrix-vector multiplications, and restriction and interpolation operators in the multigrid method. 

\subsection{Hierarchical hybrid grids}
\label{ss.hhg}

We now explain how to implement this approach more precisely. We also refer to \cite{tutorial} for a more thorough discussion, as well as \cite{ber04,ber05}.

\smallskip

We start with some definitions. We say that $\mcl T = \{K_1,\ldots, K_n\}$  is a simplicial partition of the set $U \subset \Rd$ if the following three conditions hold: (1) for every $i \in \{1,\ldots,n\}$, the set $K_i$ is a simplex in $\Rd$ (that is, the convex envelope of a set of $(d+1)$ points---a triangle in dimension $d = 2$ and a tetrahedron in dimension $d = 3$); (2) for every $i \neq j$, the interiors of $K_i$ and $K_j$ are disjoint; (3) the union $\bigcup_{i = 1}^n K_i$ is the closure of the set $U$. For convenience, we often drop the word ``simplicial'' and simply say ``partition'' instead of ``simplicial partition''. A partition can be represented as a list of nodes nodes $\{ \n_i \}_{1 \le i \le N} \subset \mathbb{R}^d$ and a list of $(d+1)$-tuples of indices that define the identity of the corner points of every simplex in the partition.  We say that two partitions $\mathcal{T}_1$ and $\mathcal{T}_2$ are nested, and write $\mathcal{T}_1 \sqsubseteq \mathcal{T}_2$, if for every $K \in \mathcal{T}_1$, there exists $T \in \mathcal{T}_2$ such that $K \subset T$.

\smallskip

We denote by $\hat K$ the standard simplex, that is, the convex envelope of the nodes $\e_0, \ldots, \e_d$, where $(\e_1, \ldots, \e_d)$ is the canonical basis of $\Rd$, and $\e_0$ is the null vector. Let $\hat{\mcl T}$ be a partition of $\hat K$, and $\mcl T_H$ be a partition of an arbitrary domain. We say that a partition $\mcl T_h$ is the \emph{locally uniform partition} associated with $(\mcl T_H, \hat{\mcl T})$, and write $\mcl T_h = \msf{lup}(\mcl T_H, \hat{\mcl T})$, if $\mcl T_h \sqsubseteq \mcl T_H$ and, for every $K \in \mcl T_H$, there exists an affine mapping $F_K : \Rd \to \Rd$ such that the image of $\hat{\mcl T}$ under $F_K$ is
$\{T \in \mcl T_h \ : \ T\subset K\}$. See Figure~\ref{f.implicit} for an illustration. Notice that the mapping $F_K$ appearing above must be such that $F_K(\hat K) = K$. Such an affine mapping is entirely specified by prescribing which nodes of $\hat K$ are sent to which nodes of $K$. 

\smallskip

Note that the locally uniform partition $\mcl T_h$ is completely specified by the knowledge of $(\mcl T_H, \hat{\mcl T})$. This allows for vast memory gains for storing the partition, since only the reference simplex $\hat K$ is meshed finely, while the global partition $\mcl T_H$ remains coarse. In addition, as discussed below, this format will be very convenient for a variety of operations, including for implementing the restriction and interpolation operators in the multigrid method.

\begin{figure}
\begin{center}
\includegraphics[scale=1]{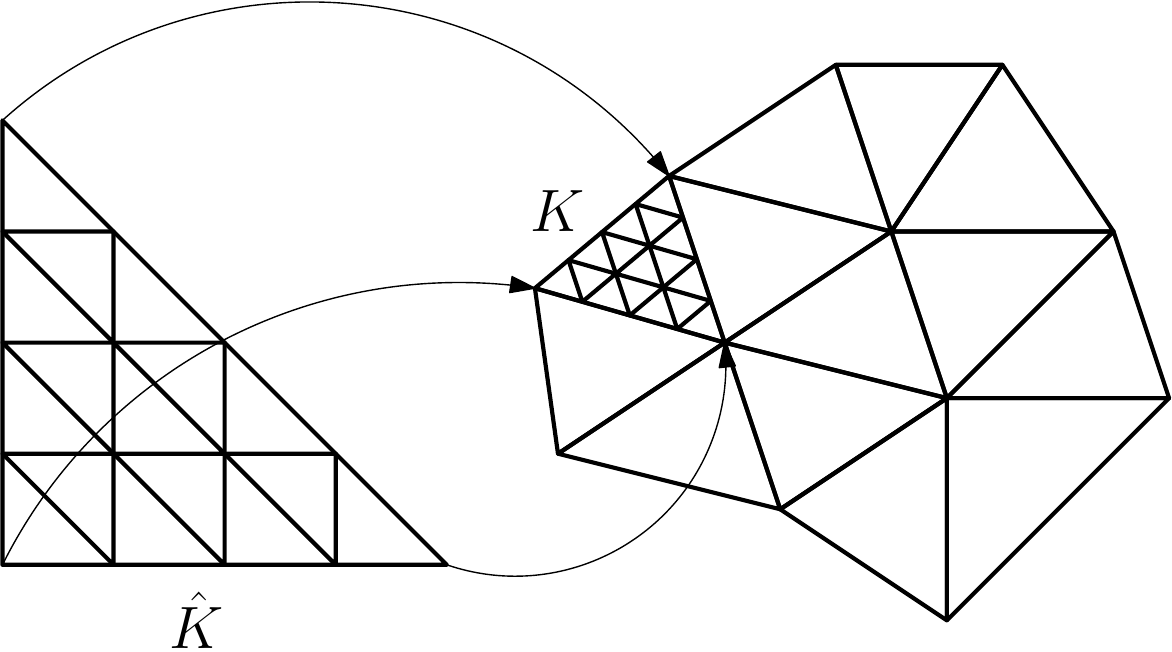} 
\end{center}
\caption{Left: the standard simplex $\hat K$ has been refined twice. Right: an unstructured coarse mesh, and the image of the twice-refined standard simplex through the affine mapping~$F_K$ for one particular coarse element $K$.}
\label{f.implicit}
\end{figure}

\subsection{Assembly of the finite-element matrix}

We proceed to define the finite-element matrix, and then describe how to store it efficiently using the structure of locally uniform partitions, under the assumption that the coefficient field is constant on each coarse element.

\smallskip

Let $\mcl T$ be a partition of the domain $U \subset \Rd$. We think of this partition as being relatively coarse, having a level of detail just sufficient to resolve the variations of the coefficient field. For clarity of exposition, we start by considering the case in which this coarse partition is not refined further. Denote by $\{ \n_i \}_{1 \le i \le N} \subset \mathbb{R}^d$ the nodes of the partition $\mcl T$. We look for an approximation of the solution of \eqref{e.model.equation} in the finite-dimensional space $V(\mcl T) \cap H^1_0(U)$, where
\begin{equation}  
\label{e.def.VT}
V(\mcl T) := \Ll\{ u \in H^1(U) \ : \ u_{\vert K} \mbox{ is affine for every $K \in \mcl T$} \Rr\} .
\end{equation}
A standard basis for $V(\mcl T)$ is formed by the nodal functions $\{ \varphi_i \}_{1 \le i \le N} \subset V(\mcl T)$, which are specified by the condition
\begin{equation*}  
\varphi_i(\n_j) = \1_{i = j}, \quad \text{for every } i, j \in \{1,\ldots, N\}.
\end{equation*}
Denoting by $\mbf x$ the vector encoding the finite-element approximation of \eqref{e.model.equation} in the basis formed by 
\begin{equation*}  
\{\varphi_i \ : \ \n_i \text{ is an interior point of } U\},
\end{equation*}
we identify $\mbf x$ as the solution of the problem
\begin{equation*}  
A \mbf x = \mbf b, \quad \mbox{where} \quad A_{ij} = \int_{U} \nabla \varphi_i \cdot \mathbf{a} \nabla \varphi_j   \quad \mbox{and} \quad \mbf  b_i =  \int_{U} \nabla \varphi_i \cdot \mathbf{a} \xi.
\end{equation*}
Notice that the size of the vectors and of the symmetric matrix appearing above is the number of nodes in the interior of $U$; this is how the null Dirichlet boundary condition is enforced. 

\smallskip

For each $K \in \mcl T$, denote by $\{\n_i^K\}_{0 \le i \le d} \subset \Rd$ the extremal points of the simplex $K$. This defines a mapping $\sigma : \mcl T \times \{0,\ldots, d\} \to \{1,\ldots N\}$, which to each $(K,i)$ associates the node number of the node $\n_i^K$ in the global ordering $\{\n_j\}_{1 \le j \le N}$. Denote by $\varphi^K_i$ the restriction to $K$ of the basis function $\varphi_{\sigma(K,i)}$. The functions $(\varphi_i^K)$ are called \emph{local shape functions}. The contribution of the element $K \in \mcl T$ to the entries of the matrix~$A$ can be represented by the matrix $A^{(K)} \in \R^{(d+1)\times (d+1)}$ such that, for every $i,j \in \{0,\ldots d\}$,
\begin{equation}  
\label{e.def.AK}
A_{ij}^{(K)} = \int_K \nabla \varphi_i^K \cdot \mathbf{a} \nabla \varphi_j^K. 
\end{equation}
The global matrix $A$ can then be reconstructed by the identity
\begin{equation}  
\label{e.rep.A}
A = \sum_{K \in \mcl T} R_K^T A^{(K)} R_K,
\end{equation}
where $R_K \in \R^{N\times (d+1)}$ is the canonical matrix representing the linear mapping
\begin{equation*}  
\Ll\{
\begin{array}{rcl}  
\R^N & \to & \R^{d+1} \\
\mbf x & \mapsto & \sum_{i = 0}^d \mbf x_{\sigma(K,i)} \e_{i+1}.
\end{array}
\Rr.
\end{equation*}
We denote the local shape functions associated with the standard simplex by $\hat \varphi_i := \varphi_i^{\hat K}$, for every $i \in \{0,\ldots d\}$, and call them \emph{reference shape functions}. In dimension $d = 2$, these reference shape functions are $x \mapsto 1-x_1-x_2$, $x \mapsto x_1$, and $x \mapsto x_2$. We also denote by $F_K : x \mapsto B_K x + v_K$ the unique affine mapping that sends the nodes $\e_0, \ldots, \e_d$ of the standard simplex to the nodes $\n_0^K,\ldots,\n_d^K$, so that in particular, $F_K(\hat K) = K$ (see Figure~\ref{f.implicit}). 
Notice that, for every $x \in K$,
\begin{equation}  
\label{e.derivatives}
\varphi_i^K(F_K(x)) = \hat \varphi_i(x), \quad \text{so that} \quad B_K^T(\nabla \varphi_i^K)(F_K(x)) = \nabla \hat \varphi_i(x).
\end{equation}
By this change of variables, for every $i,j \in \{0,\ldots d\}$, we can rewrite the integral on the right side of \eqref{e.def.AK} as
\begin{equation}  
\label{e.AK2}
A_{ij}^{(K)} = |\det B_K|\int_{\hat K} B_K^{-T} \nabla \hat \varphi_i\cdot \a B_K^{-T} \nabla \hat \varphi_j.
\end{equation}
In the case when the partition $\mcl T$ is sufficiently fine that $\a(x)$ is constant equal to~$\a^{(K)}$ when $x$ varies in $K$, we set
\begin{equation}  
\label{e.cK}
\mbf c^{(K)} := |\det B_K| \, B_K^{-1} \a^{(K)} B_K^{-T} \in \R^{d\times d},
\end{equation}
and the previous display becomes
\begin{equation}  
\label{e.AK3}
A_{ij}^{(K)} = \int_{\hat K} \nabla \hat \varphi_i\cdot \mbf c^{(K)} \nabla \hat \varphi_j.
\end{equation} 
We can expand this expression into
\begin{equation}  
\label{e.AK.as.Apq}
A^{(K)} = \sum_{p,q =1}^d \mbf c_{p,q}^{(K)} \hat A^{pq}, 
\end{equation}
where, for each $p,q \in \{1,\ldots d\}$, the matrix $\hat A^{pq} \in \R^{(d+1)\times (d+1)}$ is such that, for every $i,j \in \{0,\ldots, d\}$,
\begin{equation}  
\label{e.def.Apq}
\hat A^{pq}_{ij} := \int_{\hat K} \dr_{x_p} \hat \varphi_i \, \dr_{x_q} \hat \varphi_j.
\end{equation}
Notice that, using \eqref{e.def.AK} and \eqref{e.AK.as.Apq}, we can compute the finite-element matrix $A$ from the knowledge of $\{\mbf c^{(K)}\}_{K \in \mcl T}$ and $\{\hat A^{pq}\}_{1 \le p,q \le d}$.

\smallskip

We now generalize these observations to the case when the partition $\mcl T$ is locally uniform, say $\mcl T = \msf{lup}(\mcl T_H, \hat{\mcl T})$. We keep writing $\{\n_i\}_{1 \le i \le N}$ for the nodes of the fine partition $\mcl T$, and we denote by $\{ \hat{\n}_i \}_{1 \le i \le \hat N} \subset \hat K$ the nodes of the partition $\hat{\mcl T}$ of the standard simplex. For each $K \in \mcl T_H$, the fine partition $\mcl T$ induces a fine partition of $K$ by restriction; this partition is in fact the image of $\hat{\mcl T}$ under the mapping~$F_K$ appearing in the definition of local uniform partition. Hence, the nodes of this partition are $\n^K_i := F_K(\hat{\n}_i)$, where $i$ ranges in $\{1,\ldots, \hat N\}$. This naturally induces a mapping $\sigma : \mcl T_H \times \{1,\ldots \hat N\} \to \{1,\ldots N\}$ wich, to each $(K,i)$, associates the index of the node~$\n_i^K$ in the numbering provided by $\{\n_i\}_{1 \le i \le N}$. The mapping $\sigma$ is clearly surjective, but it is not a bijection: indeed, the nodes that belong to the boundary of multiple coarse elements are represented multiple times. On the other hand, every node that belongs to the interior of a simplex of the coarse partition has a unique representation in the form $\n_i^K$ for some $(K,i) \in \mcl T_H \times \{1,\ldots, \hat N\}$. As the partition $\hat {\mcl T}$ becomes finer and finer, the approximation $N \simeq |\mcl T_H| \, \hat N$ therefore becomes more and more accurate. (The notation $|\mcl T_H|$ stands for the number of elements in $\mcl T_H$.)

\smallskip

For each $K \in \mcl T_H$ and $i \in \{1,\ldots, \hat N\}$, we denote by $\varphi^K_i$ the restriction to $K$ of the basis function $\varphi_{\sigma(K,i)}$. The contribution of the coarse element $K \in \mcl T_H$ to the finite-element matrix can represented by the $\hat N$-by-$\hat N$ matrix $A^{(K)}$ such that \eqref{e.def.AK} holds for every $i,j \in \{0,\ldots,\hat N\}$. The relation \eqref{e.rep.A} still holds, where now $R_K \in \R^{N\times \hat N}$ is the canonical matrix representing the linear mapping
\begin{equation}  
\label{e.lin.map}
\Ll\{
\begin{array}{rcl}  
\R^N & \to & \R^{\hat N} \\
\mbf x & \mapsto & \sum_{i = 1}^{\hat N} \mbf x_{\sigma(K,i)} \e_{i}.
\end{array}
\Rr.
\end{equation}
For every $i \in \{1,\ldots, \hat N\}$, the reference shape function is defined by setting $\hat \varphi_i := \varphi_i^{\hat K}$. The identities \eqref{e.derivatives} to \eqref{e.def.Apq} still hold, the only difference being that $K$ now ranges in $\mcl T_H$ and the indices $i$ and $j$ now range in $\{1,\ldots,\hat N\}$.

\smallskip

It thus follows that the finite-element matrix associated with the locally uniform partition $\mcl T$ can be represented by storing only the set of $d$-by-$d$ matrices $\{\mbf c^{(K)}\}_{K \in \mcl T_H}$ and the set of $\hat N$-by-$\hat N$ matrices $\{\hat A^{pq}\}_{1 \le p,q\le d}$. Moreover, these matrices can be constructed directly in a straightfoward manner, without having to construct the fine partition $\mcl T$. Finally, in the practical cases we have in mind, the matrices $\hat A^{pq}$ are highly regular and have only of the order of $C\hat N$ non-zero entries. The amount of memory required to store this data is proportional to
\begin{equation*}  
d^2 \Ll( |\mcl T_H| + C \hat N \Rr).
\end{equation*}
If we were to ignore the locally uniform structure of the fine partition $\mcl T$, the cost of storing its finite-element matrix would be proportional to $N$ instead. Recalling that $N \simeq |\mcl T_H| \hat N$, we see that the semi-structured approach results indeed in a significant gain in memory usage. 

\subsection{Matrix-vector product}

Pursuing with the setting of the previous section, we now discuss how to store vectors and perform matrix-vector operations with the finite-element matrix, which we recall is represented in memory by the matrices $\{\mbf c^{(K)}\}_{K \in \mcl T_H}$ and $\{\hat A^{pq}\}_{1 \le p,q\le d}$. As discussed above, 
\begin{equation}
\label{e.def.Ax}
A = \sum_{K \in \mcl T_H} \sum_{p,q = 1}^d \mbf c_{pq}^{(K)} R_K^T \hat A^{pq} R_K,
\end{equation}
where $R_K$ is the matrix representing the linear mapping in \eqref{e.lin.map}. In view of \eqref{e.def.Ax}, instead of representing finite-element vectors as $N$-dimensional vectors, we encode them in an $\hat N$-by-$|\mcl T_H|$ array. That is, we represent each $\mbf x \in \R^N$ by an array $X$ such that the columns of $X$, denoted by $\{X(:,j)\}_{1 \le j \le |\mcl T_H|} \subset \R^{\hat N}$, are equal to the vectors $\{R_{K_j} \mbf x\}_{1 \le j \le |\mcl T_H|}$. Here we used the notation $\{K_j\}_{1 \le j\le |\mcl T_H|}$ to denote an enumeration of the (unstructured) set $\mcl T_H$. Naturally, the entries that are associated with nodes that belong to multiple coarse elements are repeated in this representation; this parallels the observation that the mapping $\sigma$ defined in the previous subsection is surjective but not bijective. 

\smallskip

The operation of $A$ onto a vector can then be evaluated in two steps: first, we compute the $\hat N$-by-$\hat N$ matrix $Y$ defined, for every $j \in \{1,\ldots,|\mcl T_H|\}$, by
\begin{equation*}  
Y(:,j) = \sum_{p,q = 1}^d \mbf c_{pq}^{(K_j)} \hat A^{pq} X(:,j).
\end{equation*}
The column $Y(:,j)$ is however not equal to the desired outcome of $R_{K_j} A \mbf x$, due to the presence of nodes that belong to multiple coarse elements. In the second step, we compute
\begin{equation}  
\label{e.matrix.vector}
(AX)(:,j) = \sum_{K_\ell \in \mcl T} R_{K_j} R_{K_\ell}^T Y(:,\ell). 
\end{equation}
In the actual implementation of this second step, we \emph{do not} need to construct the matrices $R_{K_j}$ explicitly. Instead, we implement this formula by identifying the nodes that are found at the interface between two or more elements of the coarse partition. In order to do so, we distinguish between different types of interfaces, according to whether they are to be found on faces, edges, or point vertices. (Naturally, face-type interfaces are only relevant in dimension $d = 3$.) For a more precise description of this aspect, we refer to \cite{ber04,ber05,tutorial}. 

\subsection{Multigrid method}

The geometric multigrid method is a technique for the numerical approximation of elliptic problems \cite{brenner}. It uses a sequence of nested partitions $\mcl T_n \sqsubseteq \cdots \sqsubseteq \mcl T_0$, as well as restriction and interpolation operators which allow to transfer a function defined on a given grid to a function defined on a coarser and finer grids respectively.

\smallskip

The setting of locally uniform partitions is particularly conducive to efficient implementations of the geometric multigrid method. Indeed, we first give ourselves a sequence of nested partitions of the reference element $\hat {\mcl T}_n \sqsubseteq \cdots \sqsubseteq \hat{\mcl T}_0$. 
These nested partitions are constructed as follows: we fix $\hat{\mcl T}_0 := \{\hat K\}$ to be the trivial partition, and then inductively construct $\hat{\mcl T}_{k+1}$ from $\hat{\mcl T}_k$ by adding new nodes at the middle of the edge of each element of $\hat{\mcl T}_k$, and, in dimension $d = 2$, by replacing each triangle with a partition of this triangle made of $4$ triangles, or in dimension $d = 3$, by replacing each tetrahedron with a partition of this tetrahedron made of $8$ tetrahedra~\cite{bey}. The nested partitions we use to implement the geometric multigrid method are then
\begin{equation*}  
\mcl T_n = \msf{lup}(\mcl T_H, \hat{\mcl T}_n) \sqsubseteq \cdots \sqsubseteq \mcl T_1 = \msf{lup}(\mcl T_H, \hat{\mcl T}_1) \sqsubseteq \mcl T_0 = \msf{lup}(\mcl T_H, \hat{\mcl T}_0) = \mcl T_H.
\end{equation*}
Recall that we denote by $V(\mcl T)$ the finite-element space associated with the partition~$\mcl T$, see \eqref{e.def.VT}. We start by defining interpolation and restriction operators associated with the nested partitions of the standard simplex. For each $k < n$, we define the interpolation operator $\hat {\mcl I}_k : V(\hat{\mcl T}_k) \to  V(\hat{\mcl T}_{k+1})$ to be the canonical injection. The restriction operator can then be taken as the transpose of the interpolation operator, up to a normalization constant (see \cite[Definition~6.3.1]{brenner} for more precision). Similarly, we define the interpolation operator $\mcl I_k : V(\mcl T_k) \to V(\mcl T_{k+1})$ to be the canonical injection. Recall that we represent a given vector $\mbf x_k \in V(\mcl T_k)$ as an $\hat N_k$-by-$|\mcl T_H|$ matrix~$X_k$ such that $X_k(:,j) = R^{(k)}_{K_j} \mbf x_k \in \R^{\hat N_k}$ where $\hat N_k$ is the number of vertices of the partition $\hat {\mcl T}_k$ of the standard simplex, and we wrote $R^{(k)}_{K_j}$ instead of $R_{K_j}$ to emphasize the dependency on $k$ of this operator. In this representation, we can evaluate the interpolation operator very simply by setting, for every $j \in \{1,\ldots,|\mcl T_H|\}$, 
\begin{equation*}  
(\mcl I_k X_k)(:,j) = \hat {\mcl I}_k \Ll(X_k(:,j)\Rr).
\end{equation*}
Up to a normalization constant, we wish to use the transpose of $\mcl I_k$ as our restriction operator. In view of the format in which we store elements of $V(\mcl T_{k+1})$, this is not absolutely straightfoward to compute, since it involves some amount of communication between vertices belonging to different elements of the coarse partition. We now explain how to perform this computation efficiently by reducing it to the same calculation as that arising in matrix-vector multiplication, see \eqref{e.matrix.vector}. Given a load vector $\b$ and an element $\mbf x_{k+1} \in V(\mcl T_{k+1})$, we aim to compute the residual
\begin{equation}\  
\label{e.residual.eq}
\mbf r_{k} := \mcl I_k^T \Ll( \mbf b - A \mbf x_{k+1}  \Rr) .
\end{equation}
We use the same data format to store the load vector, that is, we represent it by a family $(\mbf b^{(k+1,K)})_{K \in \mcl T_H}$ of vectors of size $\hat N_{k+1}$ such that 
\begin{equation*}  
\mbf b = \sum_{K \in \mcl T_H} \Ll(R^{(k+1)}_K\Rr)^T \, \mbf b^{(k+1,K)}.
\end{equation*}
For the model problem \eqref{e.model.equation}, this means that we set, for every $i \in \{1,\ldots, \hat N_{k+1}\}$,
\begin{equation*}  
\b^{(k+1,K)}_i := \int_K \nabla \varphi_i^{(k+1,K)} \cdot \a \xi,
\end{equation*}
where again we wrote $\varphi_i^{(k+1,K)}$ instead of $\varphi_i^K$ to make the depency on $k$ more explicit. Recall that the vector $\mbf x_{k+1}$ in \eqref{e.residual.eq} is stored in memory as an array whose columns are given by $R^{(k+1)}_K \mbf x_{k+1}$. Using also \eqref{e.rep.A}, we obtain that 
\begin{equation*}  
\mbf r_k = \sum_{K \in \mcl T_H} \mcl I_k^T \Ll(R^{(k+1)}_K\Rr)^T \, \mbf b^{(k+1,K)} - \sum_{K \in \mcl T_H} \mcl I_k^T \Ll(R^{(k+1)}_K\Rr)^T A^{(k+1,K)} R_K^{(k+1)} \mbf x_{k+1}.
\end{equation*}
Moreover, one can verify that
\begin{equation*}  
R_K^{(k+1)} \mcl I_k = \hat{\mcl I}_{k+1} R_K^{(k)}.
\end{equation*}
We thus conclude that
\begin{equation*}  
\mbf r_k = \sum_{K \in \mcl T_H} \Ll(R_K^{(k)}\Rr)^T \mbf r_k^{(K)}, \quad \text{with} \quad \mbf r_k^{(K)} := \Ll(\hat {\mcl I}_{k+1}\Rr)^T \Ll( \mbf b^{(k+1,K)}- A^{(k+1,K)}R^{(k+1)}_K \mbf x_{k+1} \Rr) .
\end{equation*}
Each term $\mbf r_k^{(K)}$ is relatively easy to compute, since $\hat {\mcl I}_{k+1}$ is an operator of moderate dimension. We have now reached a situation analogous to that in the previous section: the remaining problem is that it is not true in general that $\mbf r_k^{(K)} = R_K^{(k)} \mbf r_k$. This can be arranged by proceeding as in~\eqref{e.matrix.vector}.

\smallskip

For the smoothing steps in the multigrid method, we use a few steps of conjugate gradient descent. Finally, we use a direct solver for the coarse-grid problem. In our numerical experiments, the above-described implementation of the geometric multigrid method showed robust convergence behavior.

%
%
%
%
%
%

\section{Numerical tests}
\label{s.numerics}

In this section, we report on numerical results for the method presented in this paper. The code was written in the Julia language, and is available at this address:
\begin{equation}
\label{e.github}
\mbox{\small{\url{https://github.com/haampie/Homogenization.jl}}}
\end{equation}

In all the examples we consider, the coefficient field is $\Zd$-stationary, where $d \in \{2,3\}$ is the dimension. For convenience, we replace averages against the mask $\chi_{r_k}$ in \eqref{e.def.hatsigma} and \eqref{e.main} by averages over the cube $(-r_k,r_k)^d$. While strictly speaking, this situation is not covered by Theorem~\ref{t.main}, it is not difficult to show that the statement is still correct in this case (in fact, the argument is then somewhat simpler). For simplicity, we also fix $\ep = 0$ in \eqref{e.def.rk}. As discussed below \eqref{e.def.tdvk}, it is not difficult to modify the proof and cover this case as well, at the cost of an arbitrarily small loss of exponent in \eqref{e.main}. We also slightly modify the definition of the approximations $\td v_k$ in \eqref{e.def.tdvk}, by using a square or a cube instead of a ball for the domain: that is, for every $k \in \{0,\ldots,n\}$, we set 
\begin{equation*}  
L(k,n) := 2^{n-\frac k 2} + C_{\mathrm{bl}}(1+n) 2^{\frac k 2},
\end{equation*}
and solve for $\td v_k \in H^1_0((-L(k,n),L(k,n))^d)$ solution to
\begin{equation}
\label{e.alt.tdvk}
(2^{-k} - \nabla \cdot \a \nabla) \td v_k = 2^{-k} \, \td v_{k-1} \qquad \text{in } (-L(k,n),L(k,n))^d,
\end{equation}
with null Dirichlet boundary condition on $\dr \Ll( (-L(k,n),L(k,n))^d \Rr)$. The estimator we wish to calculate, slightly modified from \eqref{e.def.hatsigma}, is then defined by
\begin{equation}  
\label{e.alt.hatsigma}
\hat \sigma^2_n := \fint_{\Ll(-2^{n}, 2^n\Rr)^d} \Ll( -\a\xi \cdot \nabla \td v_0 + \td v_0^2 \Rr)  + \sum_{k = 1}^n 2^k \fint_{\Ll(-2^{n-\frac k 2}, 2^{n-\frac k 2}\Rr)^d}  \Ll( \td v_{k-1} \td v_k + \td v_k^2 \Rr).
\end{equation}
In order to obtain numerical approximations of the functions $\td v_k$, we use the finite-element method with hierarchical hybrid grids presented in Section~\ref{s.hhg}. In all the examples we consider, the coefficient field is piecewise constant on $z + [0,1)^d$, for every $z \in \Zd$. We thus start from a coarse partition of the domain which consists, in dimension $d = 2$, in splitting each  unit square into two triangles, or in dimension $d = 3$, in splitting each unit cube into six tetrahedra. This provides us with a coarse partition of the domain, which was denoted by $\mcl T_H$ in Section~\ref{s.hhg}. We then proceed to refine this partition iteratively by decomposing, in dimension $d = 2$, each triangle into four smaller triangles, or in dimension $d = 3$, each tetrahedron into eight smaller tetrahedra (and we do so in practice by constructing a refined partition $\mcl T_h$ of the standard simplex~$\hat K$ iteratively, which provides us with an implicit fine partition of the whole domain using the notion of locally uniform partition, see Subsection~\ref{ss.hhg}). We denote the number of iterative levels of refinement performed in this way by $N_{\mathrm{ref}}$. This defines an approximation of the quantity $\hat \sigma^2_n$ defined in \eqref{e.alt.hatsigma}, which we denote by $\hat \sigma^2(n,N_{\mathrm{ref}})$. Strictly speaking, this quantity also depends on the choice of the boundary layer size $C_{\mathrm{bl}}$, but we keep this implicit in the notation.

\smallskip

Theorem~\ref{t.main} bundles together an estimate for the mean error and an estimate for the standard deviation or our approximation $\hat \sigma^2_n$. The approximation has been set up so that both quantities are of the same order, that is, $2^{-\frac{nd}{2}}$. Additionally to this error comes the error due to the finite-element discretization: for each fixed $N_{\mathrm{ref}}$, the quantity $\hat \sigma^2(n,N_{\mathrm{ref}})$ computes an approximation of the homogenized matrix of the \emph{discretized system} with $N_{\mathrm{ref}}$ levels of refinement. While we did not prove this, it is clear that all the arguments use to prove Theorem~\ref{t.main} would remain valid for the discretized system, and thus $\sigma^2(n,N_{\mathrm{ref}})$ allows to approximate the homogenized matrix $\ahom(N_{\mathrm{ref}})$ of the discretized system with a mean error and a standard deviation that both scale like $2^{-\frac{nd}{2}}$ as $n$ tends to infinity. 
However, there is also a discrepancy between $\ahom(N_{\mathrm{ref}})$ and the homogenized matrix $\ahom$ of the continuous equation, which is manifested in our algorithm in the fact that we do not have perfect access to the solutions $\td v_{k}$ of \eqref{e.alt.tdvk}. Moreover, as explained in Subsection~\ref{ss.rough}, the rate of convergence of approximate solutions in terms of $N_{\mathrm{ref}}$ can become arbitrarily slow as the ellipticity contrast gets large.

\smallskip

As said above, we consider coefficient fields that are piecewise constant on unit cubes; more precisely, we assume that for every $z \in \Zd$, we have
\begin{equation*}  
\forall x \in z + [0,1)^d, \quad \a(x) = \b_z,
\end{equation*}
for some family $(\b_z)_{z \in \Zd}$. This family is random and constructed in the following way, given two parameters $\alpha \le \beta \in (0,\infty)$: the random variables $(\b_z)_{z \in \Zd}$ are independent; the matrix $\b_z$ is diagonal; the diagonal entries of $\b_z$, which we denote by $(\b_{z,ii})_{1 \le i \le d}$, are independent; and finally, for every $i \in \{1,\ldots, d\}$,
\begin{equation*}  
\P \Ll[ \b_{z,ii} = \alpha \Rr] = \P \Ll[ \b_{z,ii} = \beta \Rr] = \frac 1 2.
\end{equation*}
As discussed in Subsection~\ref{ss.rough}, this example is particularly interesting since it is in some sense the coefficient field which allows for the most pathological singularities in the solutions for a given ellipticity ratio $\Lambda = \beta/\alpha$. Notice that, in order to demonstrate that our numerical code is not restricted to the case when $\a(x)$ is a multiple of the identity, we have dropped this restriction here (and it would not be difficult to accomodate for matrices that are not diagonal). An additional very interesting feature of this class of examples is that it is one of the very rare cases where the homogenized matrix is known exactly: in dimension $d = 2$, it is given by $\ahom = \sqrt{\alpha \beta} \, \mathrm{Id}$ \cite[Exercise 2.10]{AKMbook}. (No such simple formula is expected to exist in dimension $d = 3$, and in fact, we are not aware of any genuinely three-dimensional coefficient field where the homogenized matrix is known exactly.)

\subsection{Two-dimensional case, moderate contrast}

\begin{figure}
\begin{center}
\includegraphics[scale=0.75, trim = 0cm 1cm 0cm 0.5cm, clip = true]{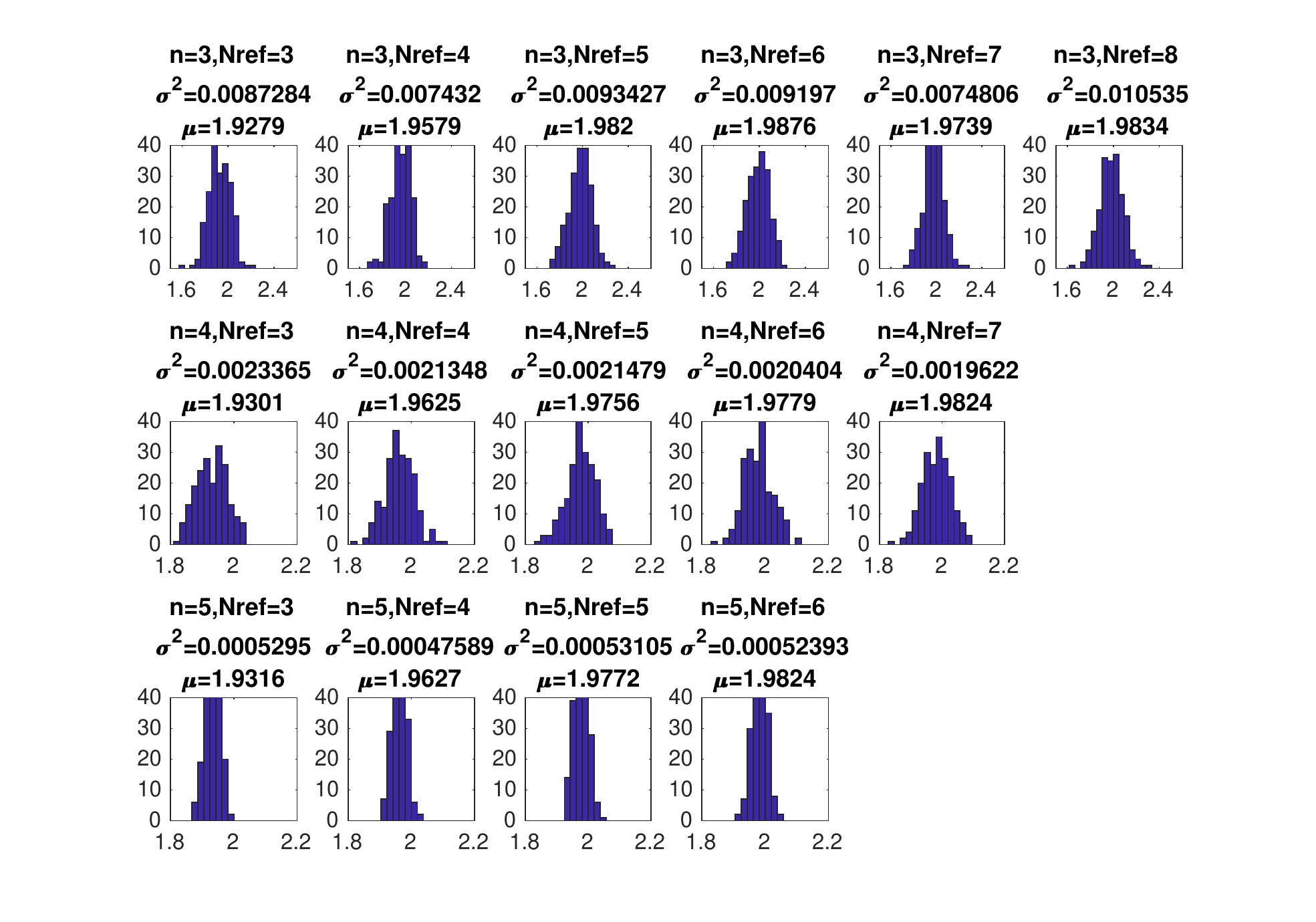}
\end{center}
\caption{Empirical distribution of $\hat \sigma^2(n,N_{\mathrm{ref}})$ when $d = 2$, $\alpha = 1$ and $\beta = 9$, for different values of $n$ and $N_{\mathrm{ref}}$. We recall that $N_{\mathrm{ref}}$ is the number of times the finite-element mesh has been refined.}
\label{fig.ex1_hist}
\end{figure}

We fix $d = 2$, $\alpha = 1$, and $\beta = 9$. We thus have in this case that $\ahom = \sqrt{\alpha \beta} \, \mathrm{Id} = 3 \, \mathrm{Id}$. Using the notation in \eqref{e.def.fintR}, we also observe that $\fint \a = 5 \, \mathrm{Id}$, and therefore we expect that $\hat \sigma^2(n,N_\mathrm{ref})$ converges to~$2$ as $n$ and $N_\mathrm{ref}$ tend to infinity. We fix the boundary layer constant $C_{\mathrm{bl}} := 4$, and plot a histogram of $\hat \sigma^2(n,N_\mathrm{ref})$ for different values of $n$ and $N_{\mathrm{ref}}$, see Figure~\ref{fig.ex1_hist}. Each histogram is obtained by sampling $200$ realizations of the estimator. For each value of $n$ and $N_{\mathrm{ref}}$, we also report the empirical mean and variance of $\hat \sigma^2(n,N_{\mathrm{ref}})$. Notice that the estimator has a bias to overestimate the value of $\ahom$, which is consistent with the fact that the remainder term~$D_n$ in the series expansion~\eqref{e.multiscale} is nonnegative, see Proposition~\ref{p.remainder} (the sign of the discretization error was not predicted theoretically).

\smallskip

From the results displayed on Figure~\ref{fig.ex1_hist}, one can check that the quantity $\hat{\sigma}^2(n = 4, N_{\mathrm{ref}} = 3)$ falls within the interval $[1.84,2.02]$ with $95\%$ probability. Taking for granted that we can estimate $\fint_\Rd \a= 5 \, \mathrm{Id}$ more easily, we obtain an estimation for $\xi\cdot \a\xi$ which falls within the interval $[2.98,3.16]$ with $95\%$ probability, the true value being $3$. This estimator thus produces a result with a relative error of $5\%$ from the true value with $95\%$ probability. It takes about 2\,s to compute this quantity on a laptop computer with 16\,Go of memory and using a single processor clocking at~2.40\,GHz. 

\smallskip

\begin{figure}
\begin{center}
\includegraphics[scale=0.5]{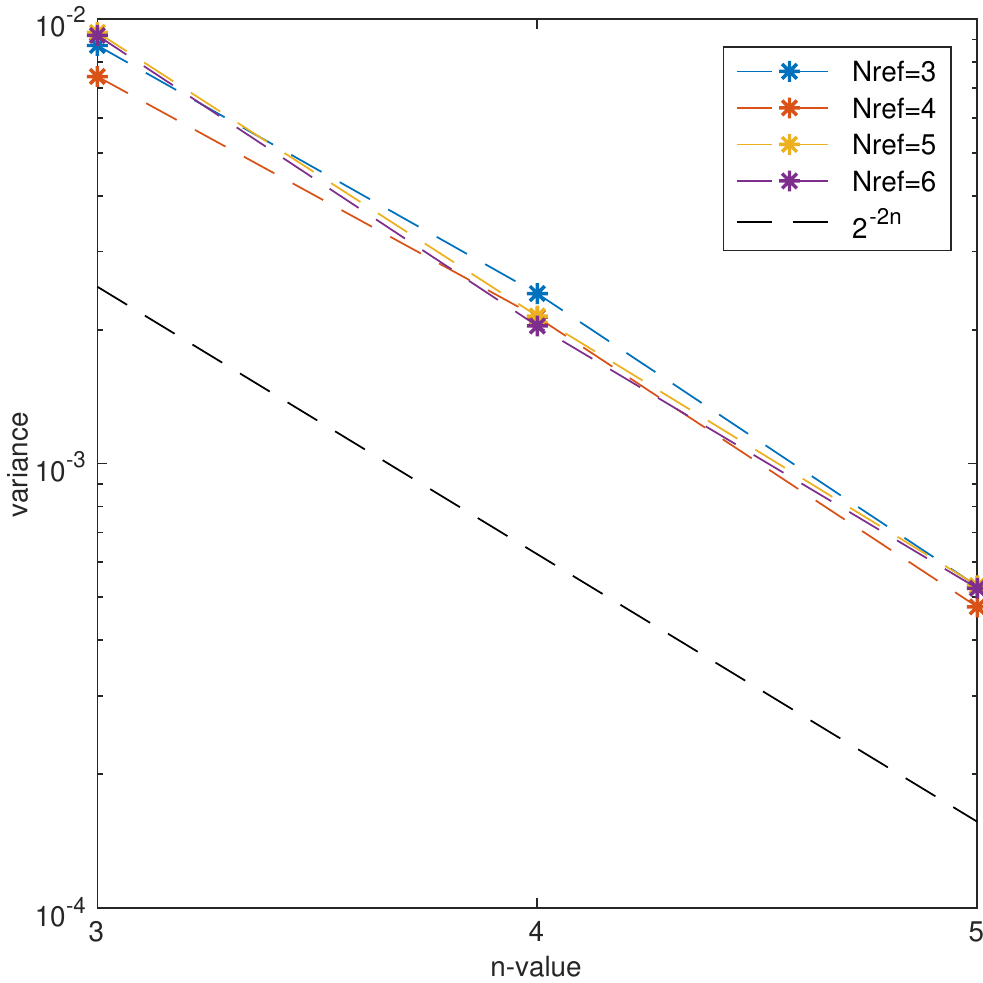} \hspace{1cm} \includegraphics[scale=0.5]{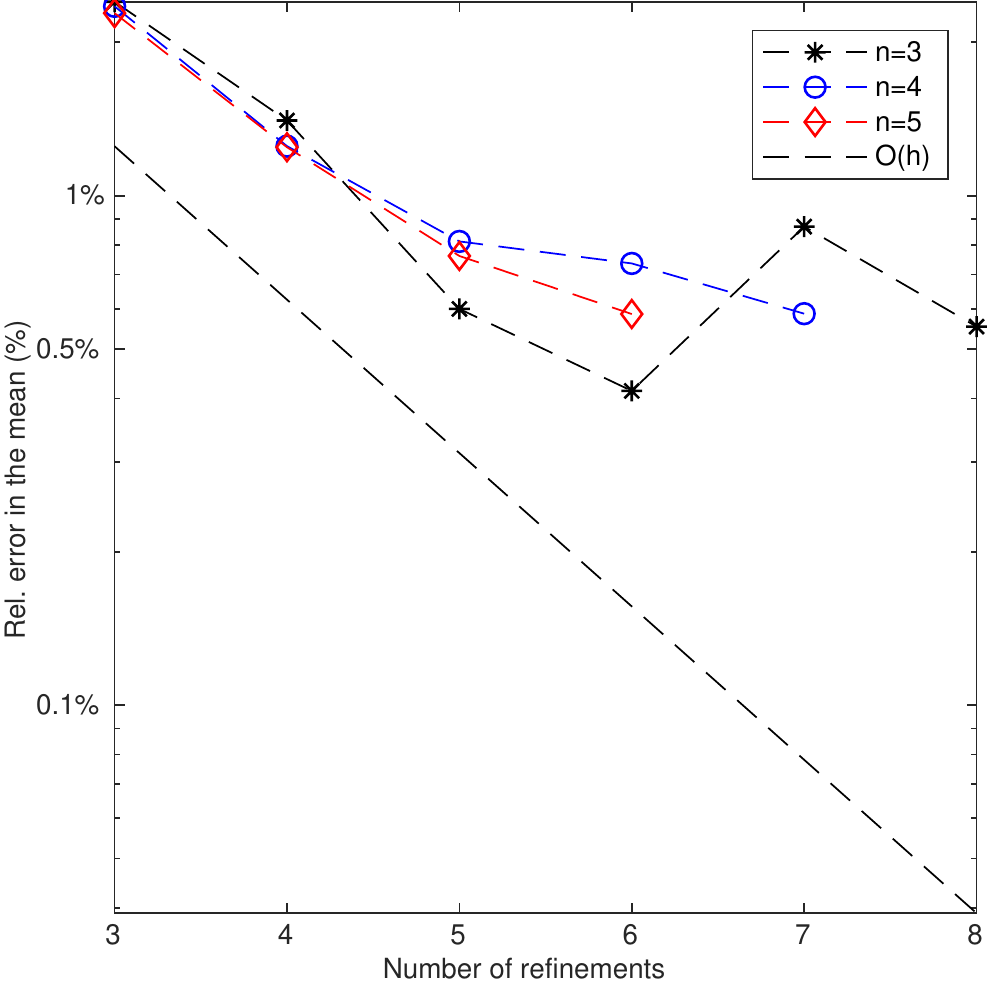}
\end{center}
\caption{Variance (left) and mean error (right) of $\hat \sigma^2(n,N_{\mathrm{ref}})$ for different values of $n$ and number of mesh refinements $N_{\mathrm{ref}}$. The variance decays approximately with the rate $2^{-dn}$ predicted by Theorem~\ref{t.main}.}
\label{fig:ex1_var_mean}
\end{figure}
\begin{figure}
\begin{center}
\includegraphics[scale=0.55, trim = 0cm 0.3cm 0cm 0cm, clip = true]{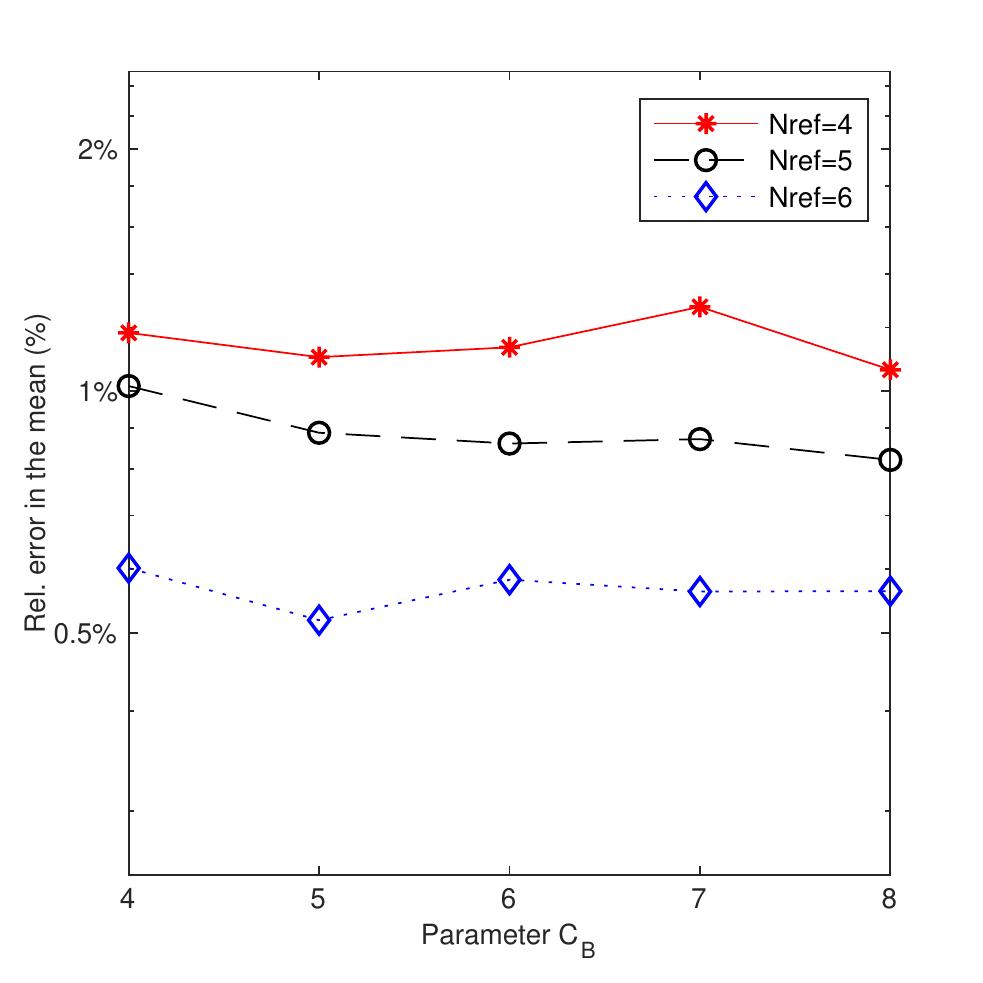} \hspace{1cm} \includegraphics[scale=0.5]{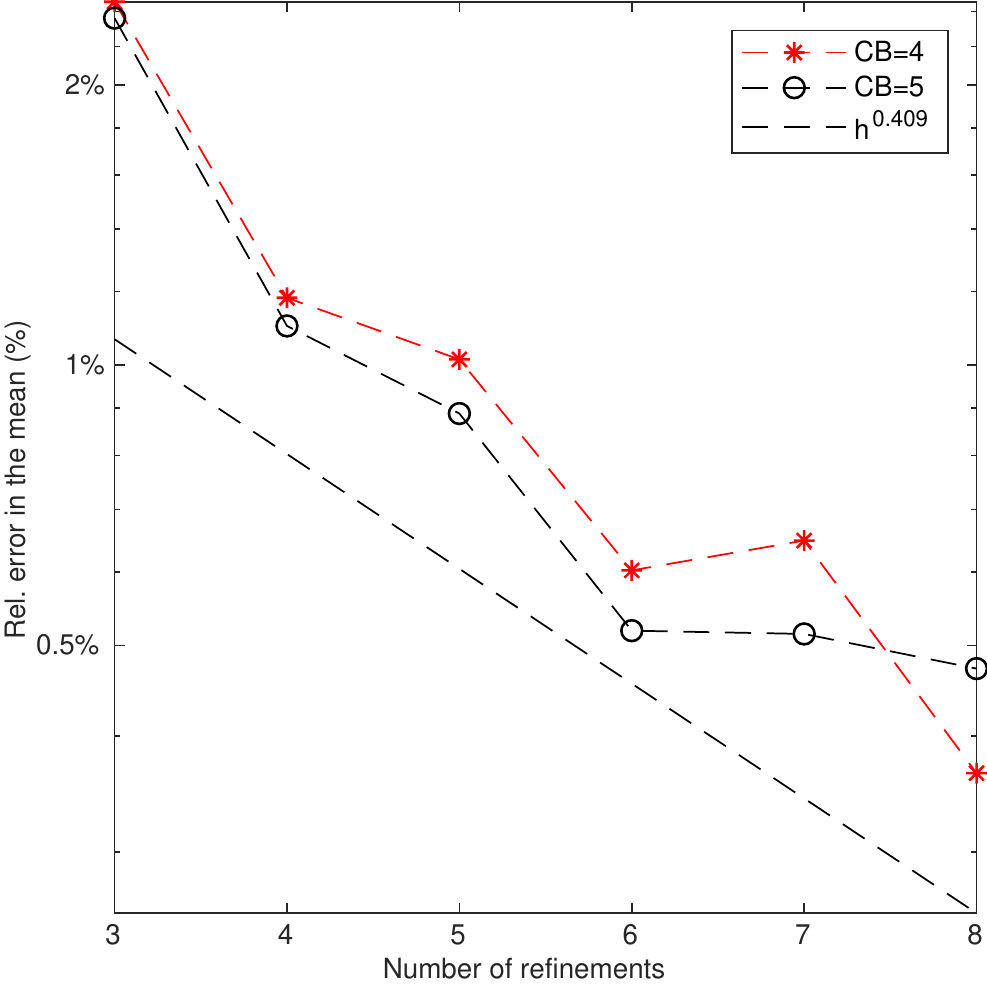}
\end{center}
\caption{Error in the mean, for $n=4$, as a function of the boundary layer constant $C_{\mathrm{bl}}$ (left), and as a function of the number of mesh refinements (right). The error is essentially independent of the value of $C_{\mathrm{bl}} \ge 4$. The dependency in $N_{\mathrm{ref}}$ is in good agreement with the predicted convergence rate in $h^\alpha$, for $\alpha \simeq 0.409$.}
\label{fig:ex1_mean2}
\end{figure}

We next investigate more precisely the scalings of the  standard deviation and mean error of $\sigma^2(n,N_{\mathrm{ref}})$. (By definition, the mean error is $|\E[\sigma^2(n,N_{\mathrm{ref}})] - 2|$ for this example). On the left frame of Figure~\ref{fig:ex1_var_mean}, we see that the variance decays like $2^{-dn} = 2^{-2n}$, as predicted by the theoretical results. On the left frame of Figure~\ref{fig:ex1_var_mean}, we display the mean error as a function of $n$ and of the number of refinements. Our theoretical arguments predict that the mean error is the sum of a term of the order of $2^{-\frac {nd}{2}} = 2^{-n}$, of the discretization error which depends on $N_{\mathrm{ref}}$, and of the boundary layer error related to the choice of $C_{\mathrm{bl}}$. We display the dependency of the mean error in these parameters more precisely on Figure~\ref{fig:ex1_mean2}, for the value $n = 4$. We see on the left frame of Figure~\ref{fig:ex1_mean2} that the choice of $C_{\mathrm{bl}} = 4$ is already sufficient to ensure that the boundary layer error is negligible compared with the discretization error. On the right frame of Figure~\ref{fig:ex1_mean2}, we observe that the discretization error decays approximately like $h^{0.409}$, where $h$ is the element size, as predicted in the discussion around \eqref{e.exponent.1.9}.

\subsection{Two-dimensional case, high contrast}
We continue with the two-dimensional setting, we also keep $\beta = 9$, but we now progressively decrease $\alpha$ in the interval $[10^{-2},1]$. More precisely, we vary $\alpha$ in twenty logarithmically equally spaced steps between the values $1$ and $10^{-2}$. We keep the parameters $n = 4$ and $N_{\mathrm{ref}} = 5$ fixed, and use $200$ samples of the estimator to compute the empirical average and standard deviation. 

\smallskip

In the code provided provided in the GitHub repository \eqref{e.github}, the choices $\alpha = 1$ and $\beta = 9$ are hard-coded, but these values can be modified by changing line 488 (or, in the three-dimensional case, line 487) of the file {\small{\textsf{src/examples/homogenized\_coefficients.jl}}}.

\smallskip

\begin{figure}
\begin{center}
\includegraphics[scale=0.5]{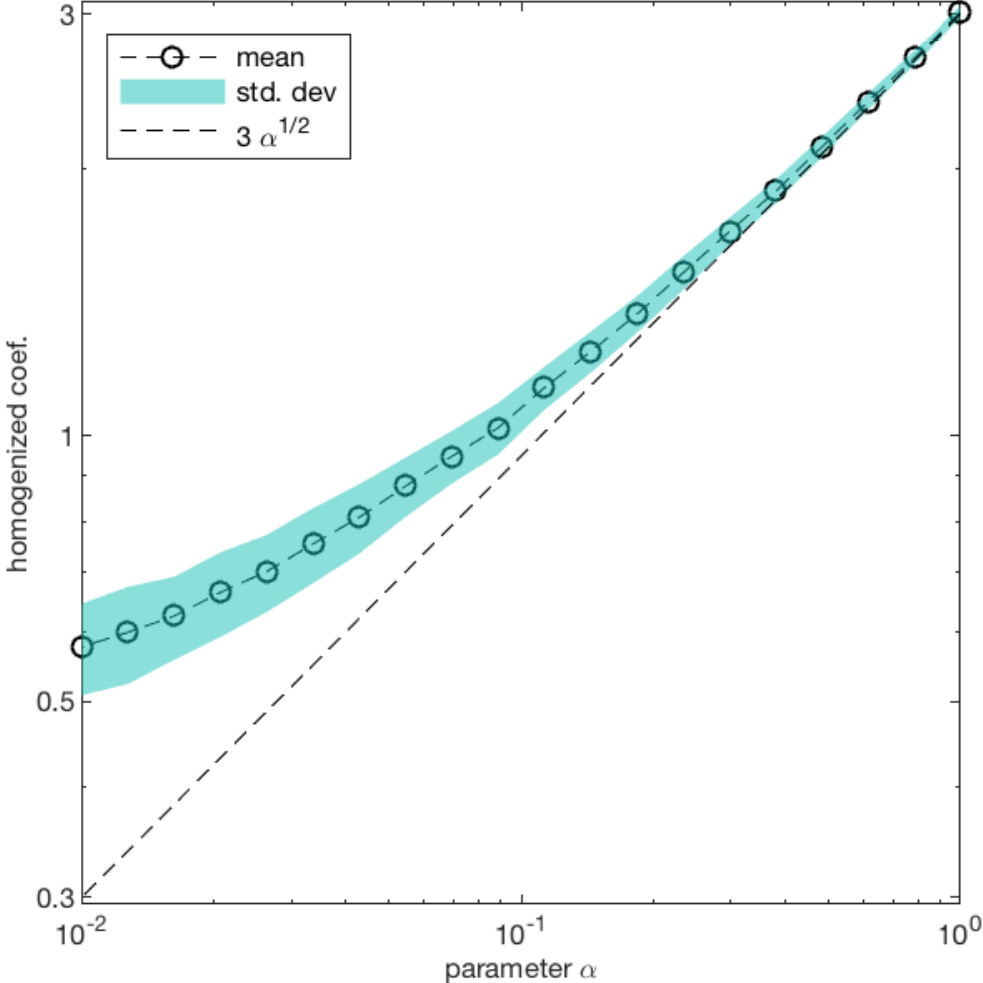} 
\hspace{1cm}
\includegraphics[scale=0.5]{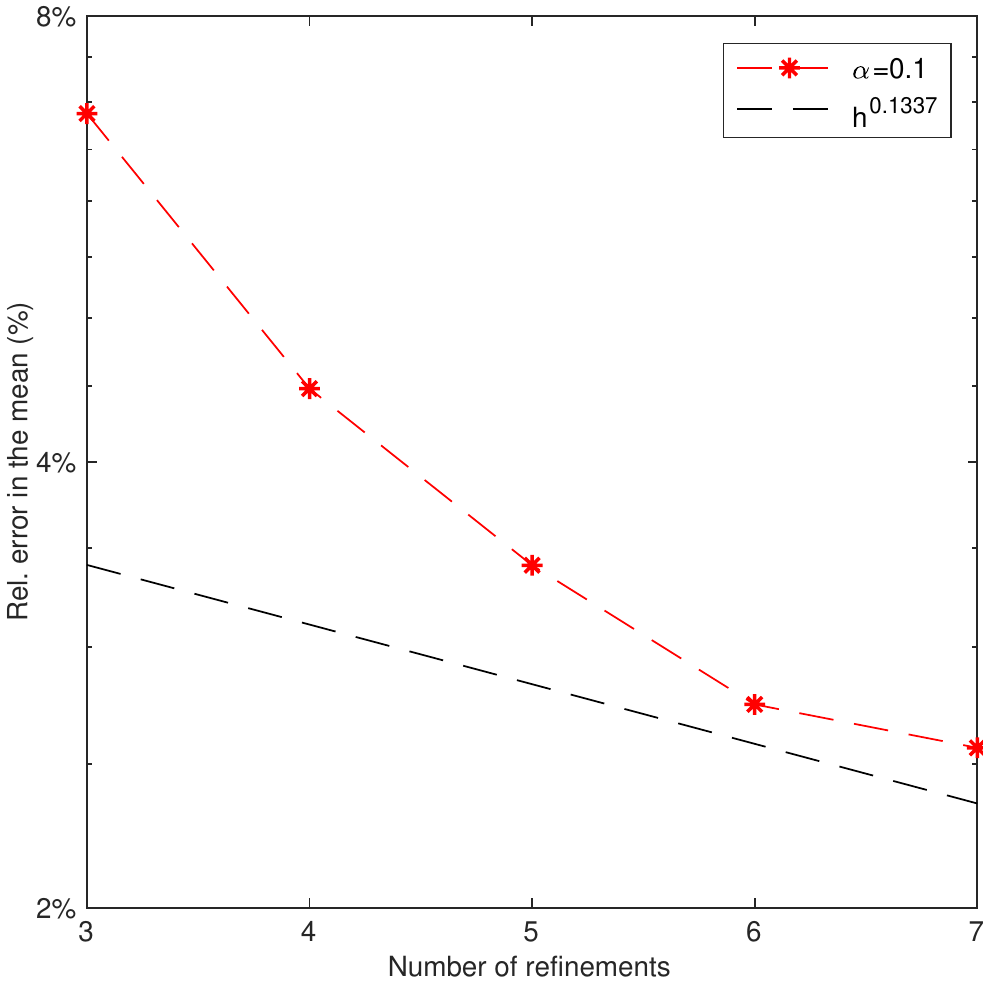}
\end{center}
\caption{Left: mean and standard deviation of $\hat \sigma^2(n = 4, N_{\mathrm{ref}} = 5)$, for different values of~$\alpha$. Notice that the logarithmic scale inflates the absolute value of the error on the left of the graph. Right: the mean error for $\alpha = 0.1$, as a function of~$N_{\mathrm{ref}}$.  For small values of $\alpha$, the finite element approximation converges very slowly, due to the singularities at the corners of the cherckerboard tiling. For $\alpha = 0.1$, we expect the asymptotic error rate to scale like $h^{0.1337}$.}
\label{fig:ex2_mean}
\end{figure}

For these twenty values of $\alpha \in [10^{-2},1]$, the left frame of Figure~\ref{fig:ex2_mean} displays the mean and the standard deviation of $\hat \sigma^2(n,N_{\mathrm{ref}})$, with the choices of $n = 4$ and $N_{\mathrm{ref}} = 5$. The estimator captures the true value of the homogenized matrix quite well, for a large span of values of $\alpha$, although relative errors become large when $\alpha$ approaches $10^{-2}$. This is in part due to the fact that the true homogenized matrix tends to zero as $\alpha$ is decrased to zero, and thus even a constant error in absolute value would translate into a relative error which blows up. A more fundamental reason for the increase of the error is that solutions become more and more singular, and thus accurate discretizations become more challenging. On the right frame of Figure~\ref{fig:ex2_mean}, we plot the relative error in the mean, for $\alpha = 0.1$ and different values of $N_{\mathrm{ref}}$. We expect the asymptotic convergence rate to scale approximately like $h^{0.1337\ldots}$, where $h$ is the size of a finite-element cell. Despite the slow asymptotic rate, a faster pre-asymptotic regime allows to bring the relative error within a few percentage points after five levels of refinement.

\subsection{Three-dimensional case, small contrast}
\begin{figure}
\begin{center}
\includegraphics[scale=0.55, trim = 0cm 1.5cm 0cm 1cm, clip = true]{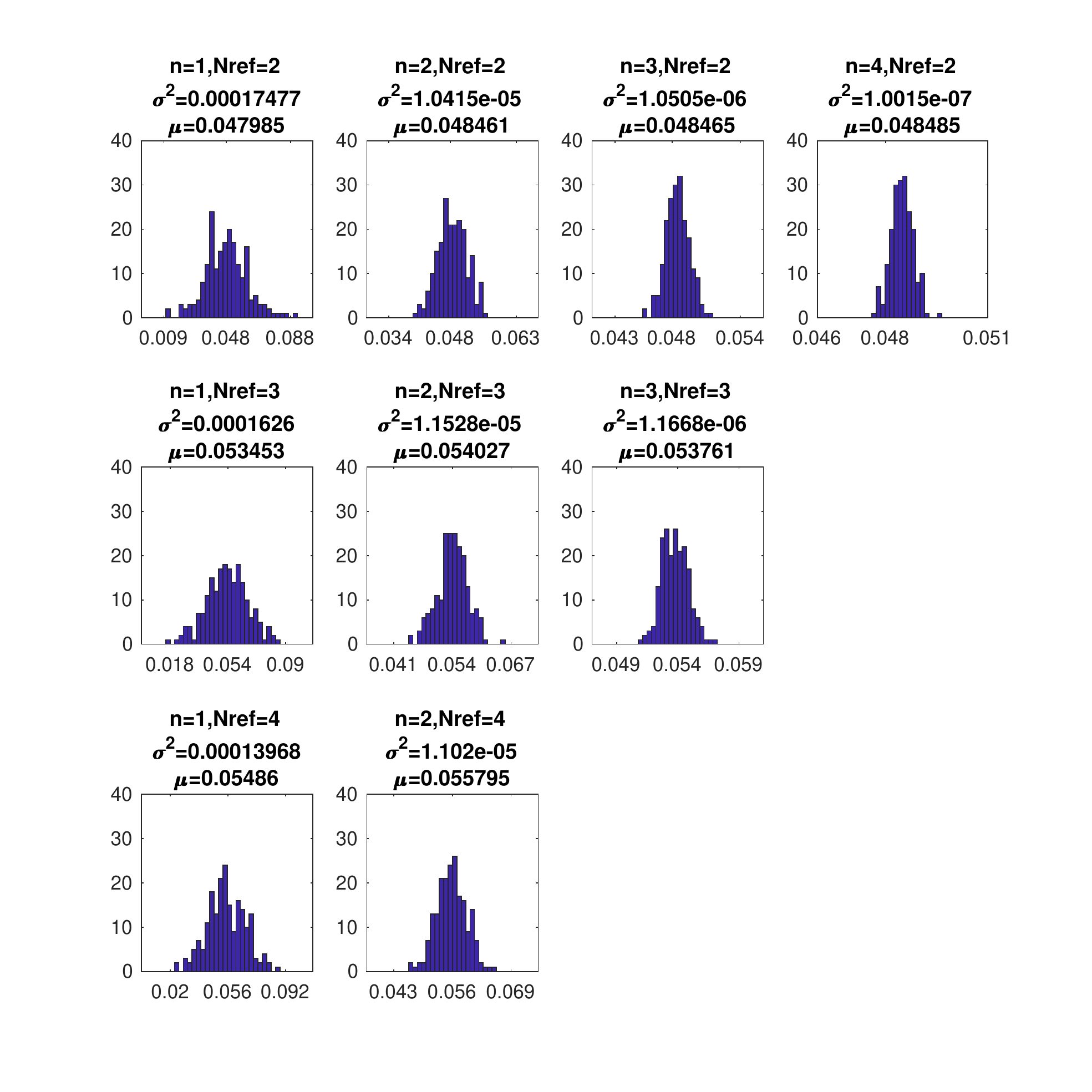}
\end{center}
\caption{Empirical distribution of $\hat \sigma^2(n,N_{\mathrm{ref}})$ when $d = 3$, $\alpha = 1$ and $\beta = 2$, for different values of $n$ and $N_{\mathrm{ref}}$.}
\label{fig.ex3_hist}
\end{figure}

We now turn to the investigation of three-dimensional problems. To further make the case that the scaling of the discretization error is strongly affected by the ellipticity contrast, we start by investigating a regime of relatively small contrast: we fix $\al = 1$ and $\beta = 2$. As in the two-dimensional case, we plot a histogram for $\hat \sigma^2(n,N_{\mathrm{ref}})$, for different values of $n$ and $N_{\mathrm{ref}}$, see Figure~\ref{fig.ex3_hist}. Each histogram is obtained by combining 200 samples of the estimator. 

\smallskip

As a rule of thumb, we expect that the approximation $\ahom \simeq \fint_{\Rd} \a$ improves as we increase the dimension and reduce the contrast. This is confirmed by the numerical results, which suggest that for the example considered, the difference $\fint_{\Rd} \a - \ahom$ is about $4\%$ of the magnitude of the homogenized matrix $\ahom$ itself. We also see that the convergence of $\hat \sigma^2(n,N_{\mathrm{ref}})$ is relatively rapid as $N_{\mathrm{ref}}$ increases. Finally, the variance decays roughly like $2^{-dn} = 2^{-3n}$, in agreement with the theoretical prediction.

\subsection{Three-dimensional case, moderate contrast}

\begin{figure}
\begin{center}
\includegraphics[scale=0.55, trim = 0cm 1.5cm 0cm 1cm, clip = true]{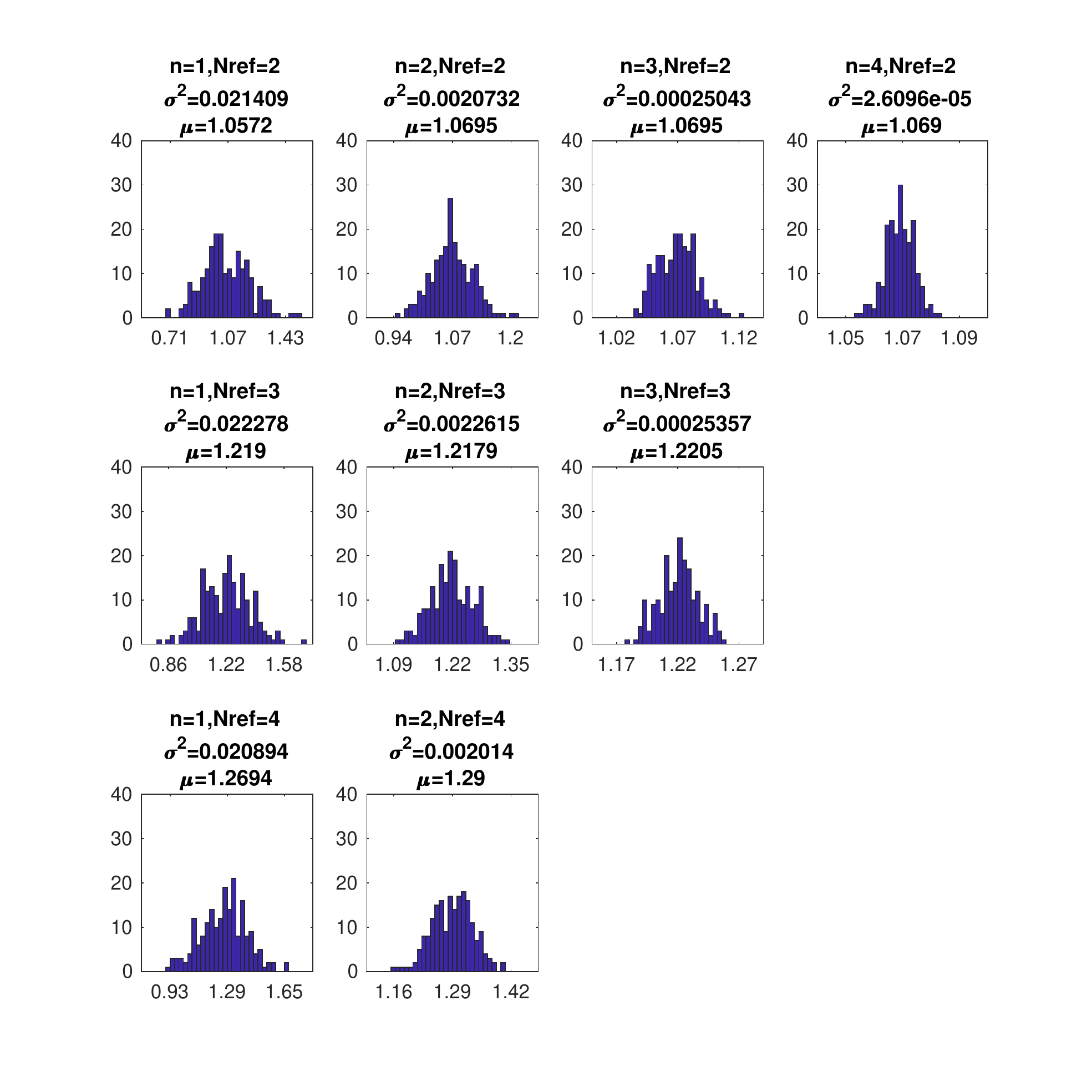}
\end{center}
\caption{Empirical distribution of $\hat \sigma^2(n,N_{\mathrm{ref}})$ when $d = 3$, $\alpha = 1$ and $\beta = 9$, for different values of $n$ and $N_{\mathrm{ref}}$.}
\label{fig.ex4_hist}
\end{figure}
We now turn to more sizable values of the ellipticity contrast, in three dimensions, fixing $\al = 1$ and $\beta = 9$. Figure~\ref{fig.ex4_hist} displays a histogram of $\hat \sigma^2(n,N_{\mathrm{ref}})$ for different values of $n$ and $N_{\mathrm{ref}}$, using 200 samples per histogram. 

\smallskip

Notice that the empirical variance of $\hat \sigma^2(n,N_{\mathrm{ref}})$ does not depend much on $N_{\mathrm{ref}}$. A linear regression based on the values for $N_{\mathrm{ref}} = 2$ suggests that this variance decays with $n$ like $C 3^{-\gamma n}$ for $\gamma \simeq 3.2$. This is in agreement with the theoretical prediction of $\gamma = d = 3$.

\smallskip

In the three-dimensional case, we are not aware of any analytic expression for the homogenized matrix. The numerical results we obtained and a naive extrapolation suggest that 
\begin{equation*}  
\fint_{\Rd}  \a - \ahom \simeq 1.35 \, \mathrm{Id}, \quad \text{and thus} \quad \ahom \simeq 3.65 \, \mathrm{Id}.
\end{equation*}
Assuming that this is correct, a $\pm 5\%$ error interval for $\ahom$ is $[3.47, 3.83]$. An average of four samples of the quantity $5-\hat{\sigma}^2(n = 2, N_{\mathrm{ref}} = 3)$ falls within this interval with probability above~$95\%$, and takes about 20\,min to compute on a laptop computer with 16\,Go of memory using a single 2.40\,GHz processor. A single sample of the quantity $5 - \hat \sigma^2(n = 2, N_{\mathrm{ref}} = 4)$ falls within the smaller interval $[3.62,3.80]$ with $95\%$ probability, and takes about 38\,min to compute with the same piece of hardware. Moreover, the computational time can be significantly reduced by optimizing on the boundary layer size.


\subsection*{Acknowledgments} AH~was partially supported by the Stenb\"{a}ck foundation and Academy of Finland project~312340.  JCM was partially supported by the ANR grants LSD (ANR-15-CE40-0020-03) and Malin (ANR-16-CE93-0003) and by a grant from the NYU--PSL Global Alliance. HS~was partially supported by Academy of Finland project~305759.

\small
\bibliographystyle{abbrv}
\bibliography{ahom}

\end{document}